\documentclass[a4paper,10 pt]{article}
\usepackage{amssymb}
\usepackage{latexsym} 
\usepackage{theorem} 
\usepackage{amsmath, amscd, bbm} 
\usepackage{enumerate} 
\usepackage[applemac]{inputenc}
\usepackage[T1]{fontenc}
\usepackage{lmodern}
\usepackage{textcomp}
\usepackage{hyperref}

\makeatletter
\theoremheaderfont{\rmfamily\scshape}
{\theorembodyfont{\itshape}
\newtheorem{thm}{Theorem}[section]
\newtheorem{prp}[thm]{Proposition}
\newtheorem{lem}[thm]{Lemma}
\newtheorem{crl}[thm]{Corollary}

\newtheorem{thmNN}{Theorem}

}
{\theorembodyfont{\upshape}
\newtheorem{definition}[thm]{Definition}
\newtheorem{definitionNN}{Definition}

\newtheorem{remark}[thm]{Remark}
\newtheorem{remarks}[thm]{Remarks}

}

\newenvironment{proof}%
{\removelastskip%
\vspace{12 pt plus1pt}
\noindent\textit{Proof.}}{\hfill $\square$\par \medskip}%

\DeclareMathOperator{\can}{can}
\DeclareMathOperator{\card}{card}

\DeclareMathOperator{\gp}{gp}
\DeclareMathOperator{\Hom}{Hom}

\newcommand{\norm}[1]{\| #1\|}
\DeclareMathOperator{\md}{md}

\DeclareMathOperator{\supp}{supp}
\def\B{\mathbb{B}}

\def\N{\mathbb{N}}

\def\Q{\mathbb{Q}}
\def\R{\mathbb{R}}
\def\s{\mathbb{S}}

\def\Z{\mathbb{Z}}
\let\@@angstroem=\AA    
\let\angstroem=\@@angstroem
\def\AA{\mathcal{A}}
\def\BB{\mathcal{B}}
\def\CC{\mathcal{C}}

\def\FF{\mathcal{F}}

\def\HH{\mathcal{H}}

\def\KK{\mathcal{K}}

\def\OO{\mathcal{O}}

\def\RR{\mathcal{R}}
\def\SS{\mathcal{S}}
\def\TT{\mathcal{T}}

\def\XX{\mathcal{X}}
\def\YY{\mathcal{Y}}
\def\ZZ{\mathcal{Z}}
\newcommand{\mono}{\rightarrowtail}
\newcommand{\incl}{\hookrightarrow}
\newcommand{\epi}{\twoheadrightarrow}
\newcommand{\iso}{\stackrel{\sim}{\longrightarrow}}
\newcommand{\eg}{e.\;g.}

\makeatother
\begin{document}
%
%
\title{A sufficient condition for finite presentability\\ of abelian-by-nilpotent groups}
\date{}
\author{Ralph Strebel}
\maketitle

\begin{abstract}
A recipe for obtaining finitely presented abelian-by-nilpotent groups is given. 
It relies on a geometric procedure 
that generalizes the construction of finitely presented metabelian groups
introduced by R. Bieri and R. Strebel in 1980. 
\end{abstract}
\smallskip

\medskip

\section{Introduction}
\label{sec:Introduction}
%
%
The theory of finitely presented soluble groups, as it is known today,  
has its origin in a startling discovery 
made by G. Baumslag and, independently, by V. N. Remeslennikov in the early 1970s:
they showed 
that \emph{every finitely generated metabelian group can be embedded into a finitely presented metabelian group} (\cite{Bau73} and \cite{Rem73a}).
This embedding is based on a fairly simple, but flexible construction;
it has been used later on by M. W. Thomson in \cite{Tho77b}
to prove  a similar embedding result for soluble linear groups
and also by O. G. Kharlampovich in her construction of a finitely presented soluble group 
with insoluble word-problem (see \cite[§3]{Kha81}).
\smallskip

\textbf{1.} The stated results might give one the impression
that the creation of new kinds of finitely presented soluble groups is only impeded by the lack of ingenuity, 
not by structural constraints.
An early result challenging this opinion was published by R. Bieri and R. Strebel in 1978.
In \cite{BiSt80}, these authors strengthened the earlier result 
and complemented it by a very general method for manufacturing finitely presented metabelian groups.
The two results, taken together,
lead to a characterization of the finitely presented metabelian groups 
among the finitely generated metabelian groups.

The method for producing finitely presented metabelian groups admits of an extension 
that yields \emph{finitely presented abelian-by-nilpotent groups}.
I found this generalization in 1981,  but did not publish it, as its utility was far from clear: 
the new method did not seem to lead to an embedding result for finitely generated abelian-by-nilpotent groups 
into finitely presented groups of the same kind, 
and the hypotheses needed to carry out the construction seemed to be far from being necessary.
\smallskip
 
\textbf{2.} 
The difficulties encountered by the new method in producing interesting finitely presented abelian-by-nilpotent groups
have later been explained by C. Brookes, J. R. J. Groves, J. E. Roseblade and J. S. Wilson 
in a series of papers in the late 1990s:
 in \cite{BRW97}, Brookes, Roseblade and Wilson showed 
 that a \emph{finitely presented abelian-by-polycyclic group $G$} is necessarily nilpotent-by-nilpotent-by-finite; 
 in \cite{Bro00} then, this conclusion is sharpened  
 to \emph{$G$ is nilpotent, by nilpotent of class at most two, by finite}.
  
 The question now arises whether Brookes' conclusion can be sharpened further.
 In answering this question,
 the old result of 1981 has turned out to be helpful: 
J. R. J. Groves and the author use it in \cite{GrSt11} 
to construct, given a finitely generated nilpotent group $Q$ of class two, 
 a finitely generated $\Z{Q}$-module $A$ enjoying the properties
 that $A$ is the Fitting subgroup of the split extension $G = Q \ltimes A$ 
 and that $G$ admits a finite presentation.
\smallskip

\textbf{3.} 
This paper is an updated and abridged version 
\footnote{I have omitted the construction of finitely presented groups $G$
which are extensions $A \triangleleft G \epi Q$ of an abelian normal subgroup $A$ by a nilpotent group $Q$
that do not split.}
of the unpublished article \cite{Str81b}.
Its contents and layout are as follows:
in Section \ref{sec:Invariant-Sigma0}, the invariant $\Sigma^0$ is introduced. 
Its definition can be stated for a finitely generated, but otherwise arbitrary,  group $Q$ 
and a finitely generated $\Z{Q}$-module
and reads:
\[
\Sigma^0(Q;A) = \{ [\chi] \in S(Q) \mid A \text{ is fg over the monoid ring }\Z{Q}_\chi \}.
\]
(See section \ref{ssec:Sigma0-preliminaria}
for the definitions of the sphere $S(Q)$, the submonoids $Q_\chi$ and the rays $[\chi]$).
Some properties of the invariant can be established for arbitrary groups,
but for most of the results proved in this paper stronger assumptions are needed.
This is, in particular, true for the hypothesis imposed on the module $A$ 
that enables one to show that the split extension $G = Q \ltimes A$ admits of a finite presentation.
This hypothesis, called \emph{tameness}, is defined and investigated in Section \ref{sec:Tame-modules}; 
it presupposes that $Q$ is nilpotent and it involves a central series
$\{1\} = Q_{k+1} < Q_k < \cdots < Q_2 < Q_1 = Q$ of $Q$.

In the special case where $Q$ is nilpotent of class 2 and the series is of length 2  with $Q_2 = [Q,Q]$,
the definition of tameness reads as follows:
\begin{definitionNN}
Let $A$ a finitely generated $\Z{Q}$-module and  $\AA$ a finite generating set of $A$.
We call  $A$ \emph{tame} if the invariants of $Q$ and of $Q' = [Q,Q]$ satisfy the following requirements:
\begin{align*}
\Sigma^0(Q; A) &\cup  -\Sigma^0(Q;A )= S(Q)  \text{ and }\\
\Sigma^0(Q'; \AA \cdot \Z{Q'}) &\cup  -\Sigma^0(Q'; \AA \cdot \Z{Q'}) = S(Q').
\end{align*}
\end{definitionNN}
The notion of tame modules in hand,
the main result can be stated like this: 
\setcounter{section}{4}
\begin{thmNN}
\label{thm:Main-result}
Let $Q$ be a finitely generated nilpotent group and $A$ a finitely generated $\Z{Q}$-module.
 If $A$ is tame the semi-direct $Q \ltimes A$ has a finite presentation.
\end{thmNN}
\setcounter{section}{1}
This result is established in Section \ref{sec:Construction-groups}.

The paper concludes with some examples illustrating the use of Theorem \ref{thm:Main-result}
and comments on the necessity of the requirement that the module $A$ be tame.
\smallskip

\textbf{4.} I thank John Groves for bringing to my attention the problem studied in \cite{GrSt11}.
This problem provided the impetus for publishing the old results set forth in this article.
I am also grateful to the referee for his or her unusually thorough reading of the paper 
and the helpful list of suggestions; 
they lead to a number of improvements of the article, both in content and in exposition.
%
%
%
\section{The invariant $\Sigma^0(Q;A)$}
\label{sec:Invariant-Sigma0}
%
%
The definition of the invariant $\Sigma^0(Q;A)$ for a finitely generated group $Q$ 
is a straightforward generalization  of the definition of $\Sigma_A(Q)$ introduced in \cite{BiSt80},
where $Q$ is a finitely generated abelian group.
The properties of $\Sigma^0(Q;A)$, however, are harder to establish 
and are less satisfactory than those of $\Sigma_A(Q)$.
In order not to overburden this preliminary section,
$Q$ will be a finitely generated, but otherwise arbitrary, group in section \ref{ssec:Sigma0-preliminaria}
and at the beginning of section \ref{prp:Sigma0-criterion},
but later on it will be assumed to be polycyclic;
this restriction leads to shaper and more easily stated results.
(In sections \ref{sec:Tame-modules} and \ref{sec:Construction-groups}
the group $Q$ will actually be nilpotent.)
%
%
\subsection{Definition and first properties}
\label{ssec:Sigma0-preliminaria}
%
We define the invariant $\Sigma^0(Q;A)$, establish some of its easy properties
and then introduce a function $v_{\chi} \colon \Z{Q} \to \R \cup \{\infty\}$ 
whose properties are similar to those of a valuation of a field.
%
\subsubsection{The sphere $S(Q)$ and the invariant $\Sigma^0(Q;A)$}
\label{sssec:Sigma0-definition}
%
Let $Q$ be a finitely generated group.
The group $Q$ gives rise to a sphere $S(Q)$ obtained like this:
the homomorphisms $\chi \colon Q \to \R$ of $Q$ into the additive group of $\R$
form a finite dimensional vector space $\Hom(Q, \R)$ over the reals
and each non-zero homomorphism $\chi$ of $Q$ gives rise to a subset 
\[
Q_{\chi} = \{q \in Q \mid \chi(q)  \geq 0\} 
\]
of $Q$.
This subset is actually a submonoid of $Q$ and it does not change  
if $\chi$ is replaced by a positive multiple of itself.
The space of submonoids $Q_\chi$ is therefore parametrized by the rays $[\chi]$ emanating from the origin;
they are the points of the sphere 
\begin{equation}
\label{eq:Character-sphere}
S(Q) = \{ [\chi] \mid \chi \in \Hom(Q, \R) \smallsetminus \{0\} \}.
\end{equation}
For later use, we introduce also  a collection of subspheres of $S(Q)$:
given a subgroup $K$ of $Q$ we set
\begin{equation}
\label{eq:Subsphere-character-sphere}
S(Q, K) = \{ [\chi] \in S(Q) \mid \chi(K) = \{0 \}  \}.
\end{equation}

The invariant $\Sigma^0(Q;A)$ is a subset of the sphere $S(Q)$;
it depends on $Q$ and on a finitely generated $\Z{Q}$-module $A$.
The module $A$ can be regarded as module over the various monoid rings $\Z{Q_\chi}$,
but $A$ may not be finitely generated over some of them.
The invariant $\Sigma^0(Q;A)$ collects those submonoids 
over which $A$ remains finitely generated:
\begin{equation}
\label{eq:Definition-SigmaA(Q)}
\Sigma^0(Q;A) = \{ [\chi] \in S(Q) \mid A \text{ is fg over the monoid ring }\Z{Q}_\chi \}.
\end{equation}
\begin{remarks}
\label{remarka:Origin-definition-sphere-invariant}
a) In \cite{BiSt80}, a subset $\Sigma_A(Q)$ is studied,
with $Q$ a finitely generated \emph{abelian} group.
In \cite{Str81b} the generalization given in the above is put forward and analyzed.
Later on, higher dimensional analogues of $\Sigma^0$ were defined and investigated 
by R.~Bieri and B.~Renz in \cite{BiRe88}.

b) The modules in this paper will be \emph{right} modules, 
in contradistinction to the modules considered in  \cite{BiSt80} and in \cite{BiRe88}.
Every right $Q$-module $A$ can, of course, be turned into a left $Q$-module $A_{s}$ 
by declaring that $q \star a = a \cdot q^{-1}$ 
and definition \ref{eq:Definition-SigmaA(Q)} has an obvious interpretation for left $\Z{Q}$-modules.
The two subsets are then related by $\Sigma^0(Q; A_{s}) = - \Sigma^0(Q;A)$.

c) In the context of the subsets Sigma it has become customary to refer to 
the homomorphisms into the additive group of $\R$  as \emph{characters}.
\end{remarks}

\subsubsection{Elementary properties of $\Sigma^0$}
\label{sssec:Basic-properties-Sigma0}

We continue with four simple properties of $\Sigma^0$.
In the first of them,  $Q$ is fixed and $A$ is varied.
\begin{lem}
\label{lem:ses-modules}
For every extension $A_{1} \mono A \epi A_{2}$  of finitely generated $\Z{Q}$-modules 
the chain of inclusions
\begin{equation}
\label{eq:Sigma0-extension-modules}
\Sigma^0(Q;A_{1}) \cap \Sigma^0(Q;A_{2}) \subseteq \Sigma^0(Q;A) \subseteq \Sigma^0(Q;A_{2})
\end{equation}
holds. 
If the extension splits the first of these inclusions is actually an equality.
\end{lem}

\begin{proof}
If $[\chi]$ is contained in the intersection,
both $A_{1}$ and $A_{2}$ are finitely generated over $\Z{Q_{\chi}}$,  
whence $A$ has the same property.
So $[\chi] \in \Sigma^0(Q; A)$ by definition.
The inclusion $\Sigma^0(Q;A) \subseteq \Sigma^0(Q;A_{2})$ is proved similarly. 
If the extension splits, the module $A_{1}$ is a quotient module of $A$ 
and so the inclusion $\Sigma^0(Q;A) \subseteq \Sigma^0(Q;A_{1})$ holds
by the second inclusion in \eqref{eq:Sigma0-extension-modules}.
It follows that the first inclusion in \eqref{eq:Sigma0-extension-modules} is then an equality.
\end{proof}

The second result compares $\Sigma^0(Q;A)$ with $\Sigma^0(P;A)$, 
the subgroup $P$ being of finite index in $Q$. 
\begin{lem}
\label{lem:Subgroup-finite-index-Sigma0}
Suppose $\iota \colon  P \hookrightarrow Q$  is the inclusion of a subgroup $P$ of  finite index into $Q$.
Then the morphism $\iota^* \colon S(Q) \mono S(P)$ induced by the inclusion of $\iota$
maps $\Sigma^0(Q;A)$ bijectively onto $\Sigma^0(P; A)  \cap S(P, P\cap Q')$. 
\end{lem}

\begin{proof}
Notice first that $A$ is finitely generated over the subring $\Z{P} \subset \Z{Q}$. 
Consider now a non-zero character $\chi \colon  Q \to \R$.
Then $\chi(P) \neq \{0\}$;
so there exists a finite subset  $\TT \subset Q$ which contains the unit element 1,
represents the homogeneous space $P\backslash Q$ and satisfies $\chi(\TT) \leq 0$. 
Set $\psi = \chi \circ \iota$.

Assume now  $A$ is finitely generated over $\Z{Q_{\chi}}$,  say by the finite set $\AA$.
For every $q \in Q_{\chi}$ there exists $p \in P$ and $t \in \TT$ such that $q = p \cdot t$.
Then $\chi(p) \geq 0$, for $\chi(\TT) \leq 0$, and so $p \in P_{\psi}$.
It follows that $Q_{\chi}$ is contained in the union of the subsets $ P_{\psi} \cdot t$ with $t \in \TT$,
and so $A$ is generated  over $P_{\psi}$ by the finite set $\AA \cdot \TT$.
This reasoning proves the inclusion
\[
\iota^*(\Sigma^0(Q;A)) \subseteq \Sigma^0(P; A)  \cap S(P, P\cap Q').
\]
Its converse is clear.
\end{proof}

There is a variant of Lemma \ref{lem:Subgroup-finite-index-Sigma0}
where one starts with a module $B$ of the smaller group, induces it up to a module $A$ of the larger group
and then compares the invariants $\Sigma^0(P; B)$ and $\Sigma^0(Q;A)$.
If $P$ is normal in $Q$, the following result holds:
\begin{lem}
\label{lem:Subgroup-finite-index-Sigma0-induced-module}
Suppose $\iota \colon  P \hookrightarrow Q$  is the inclusion of a normal subgroup $P$ of finite index into $Q$
and $B$ is a finitely generated $\Z{P}$-module. 
Set $A = B \otimes_{P} \Z{Q}$.
Then the morphism $\iota^* \colon S(Q) \mono S(P)$
maps $\Sigma^0(Q;A)$ bijectively onto $\Sigma^0(P; B)  \cap S(P, P\cap Q')$. 
\end{lem}

\begin{proof}
Lemma \ref{lem:Subgroup-finite-index-Sigma0} implies that image of $\Sigma^0(Q;A)$
under the isomorphism $\iota^*$ coincides with the intersection  $\Sigma^0(P;A) \cap S(P, P \cap Q')$;
here the $Q$-module $A= B \otimes_{P} \Z{Q}$ 
is to be considered as an $P$-module via restriction. 
As such, 
it is a finite direct sum of modules of the form $B \otimes q$ with $q \in Q$.
Fix $q \in Q$
and let $\alpha_q$ denote the automorphism of $H$ that sends $p$ to $qpq^{-1}$.
An element $p \in P$ acts on a summand $B \otimes q$ as follows:
\[
(b \otimes q) p = (b \otimes q p q^{-1}) \cdot q = b \cdot \alpha_q(p) \otimes  q.
\]

Consider now a non-zero character $\chi \colon Q \to \R$.
We claim it coincides on $P$ with $\chi \circ \alpha_q$;
indeed, if $p \in P$ then 
$(\chi \circ \alpha_q)(p) = \chi (qpq^{-1})  = \chi(p)$.
It follows that $A \otimes q$ is finitely generated over $P_\chi$ if and only if $A$ has this property.
But if so, 
\[
\Sigma^0(P; B) \cap S(P, Q' \cap P) = \Sigma^0(P; B \otimes q) \cap S(P, Q' \cap P) 
\]
for every $q\in Q$. 
The claim now follows from the addendum to Lemma \ref{lem:ses-modules}.
\end{proof}

In Lemma \ref{lem:Subgroup-finite-index-Sigma0} 
one pulls back the action of $Q$ on the $\Z{Q}$-module $A$ to a subgroup $P$ of $Q$
and asks how the invariant $\Sigma^0(P;A)$ is related to $\Sigma^0(Q;A)$.
An analogous question can be asked for an epimorphism $\pi \colon \tilde{Q} \epi Q$.
Its answer is given by
\begin{lem}
\label{lem:Quotient-groups-Sigma0}
Suppose $\pi \colon  \tilde{Q}\epi Q$  is an epimorphism of a finitely generated group $\tilde{Q}$ onto $Q$
and $A$ is a finitely generated $\Z{Q}$-module.
Let $\tilde{A}$ denote the $\Z{\tilde{Q}}$-module obtained from $A$ by pulling back the action via $\pi$.
Then 
\begin{equation}
\label{eq:Sigma0-pulling-back-action}
\Sigma^0(\tilde{Q};\tilde{A}) = \pi^{*}(\Sigma^0(Q;A)) \cup S(\tilde{Q}, \ker \pi)^c.
\end{equation}
In this formula $\pi^{*}$ denotes the embedding $S(Q) \mono S(\tilde{Q})$ induced by  $\pi$.
\end{lem}

\begin{proof}
Notice first that $\tilde{A}$ is finitely generated over $\tilde{Q}$.
Consider now a non-zero character $\tilde{\chi}\colon \tilde{Q}\to \R$;
two cases arise.
If $\tilde{\chi}$ vanishes on the kernel of $\pi$ it is a pull-back of some character $\chi$ of $Q$,
and $\tilde{A}$ is finitely generated over $\Z\tilde{Q}_{\tilde{\chi}}$ 
if, and only if, $A$ is finitely generated over $\Z{Q}_{\chi}$. 
If, on the other hand, $\tilde{\chi}$ does not vanish on  $N = \ker \pi$
there exists an element $t \in N$ with $\tilde{\chi}(t) > 0$. 
Then $t$ acts by the identity on $\tilde{A}$. 
Since $\tilde{A}$ is finitely generated over $\Z{\tilde{Q}}$
and as $\tilde{Q} = \bigcup\nolimits_{j \in \N} \tilde{Q}_{\tilde{\chi}} \cdot t^{-j}$,
the module $\tilde{A}$ is therefore finitely generated over $\Z{\tilde{Q}_{\tilde{\chi}}}$.
\end{proof}
%
\subsubsection{Valuations extending characters}
\label{sssec:Valuation-extending-character}
%
The second topic of this preliminary section is a function associated to a character $\chi$.
The function is defined on the group ring $\Z{Q}$ of $Q$ 
and has properties similar to those of valuations of a field:
\begin{definition}
\label{definition:Naive-valuation-on-group-ring}
Given a character $\chi \colon Q \to \R$,
define $v_{\chi}\colon \Z{Q} \to \R \cup \{\infty \}$ by the formula
\begin{equation}
\label{eq:Naive-valuation-on-group-ring}
v_{\chi}(\lambda) = 
\begin{cases}
\min\{\chi(q) \mid q \in \supp(\lambda) \}& \text{ if } \lambda \neq 0,\\
\infty &\text{ if } \lambda =0.
\end{cases}
\end{equation}
Here $\lambda$ is viewed as a function  $Q \to \Z$ with finite support $\supp(\lambda)$.
The function $v_{\chi}$ will be called the  \emph{naive valuation extending} $\chi$.
\end{definition}

Basic properties of the valuation $v_{\chi}$ are collected in Lemma 
\ref{lem:Properties-valuation-vsubchi};
the proofs are straightforward and hence omitted.
\begin{lem}
\label{lem:Properties-valuation-vsubchi}
The function $v_{\chi}$ enjoys the following properties:
\begin{align}
v_{\chi}(\lambda + \mu) &\geq \min \{v_{\chi}(\lambda), v_{\chi}(\mu)\}
&\text{ for all } &(\lambda, \mu) \in \Z{Q}^2,
\label{eq:Property-addition}\\
v_{\chi}(\lambda \cdot \mu) &\geq v_{\chi}(\lambda) +v_{\chi}(\mu)
&\text{ for all } &(\lambda, \mu) \in \Z{Q}^2,
\label{eq:Property-multiplication}\\
v_{\chi}(\lambda \cdot q) &= v_{\chi}(\lambda) + \chi(q)
&\text{ for all } &( \lambda, q) \in  \Z{Q} \times Q .
\label{eq:Property-multiplication-group-element}
\end{align}
If the group ring $\Z{Q}$ has no zero-divisors,
the two terms in \eqref{eq:Property-multiplication} coincide.
\end{lem}
%
\subsection{$\Sigma^0$-criteria}
\label{ssec:Sigma0-criterion}
%
The $\Sigma^0$-criteria imply that $\Sigma^0(Q;A)$ is an open subset of $S(Q)$;
they allow one also to prove that a given point $[\chi]$ lies in $\Sigma^0(Q;A)$.
In Sections \ref{sec:Tame-modules} and \ref{sec:Construction-groups},
a criterion valid for polycyclic groups will be needed.
In the proof of this criterion one uses, however, arguments,
that constitute, in essence, a proof of a criterion valid for arbitrary finitely generated groups.
For the sake of clarity,
I give therefore first a criterion for arbitrary groups 
and then a refinement for polycyclic groups.
%
\subsubsection{First criterion}
\label{sssec:Sigma0-criterion-arbitrary-groups}
%
%
\begin{prp}
\label{prp:Sigma0-criterion}
Let $\AA$ denote a finite set generating the $\Z{Q}$-module $A$ and let $r$ be a non-negative real number.
Then the following statements are equivalent for every non-zero character $\chi \colon Q \to \R$: 
\begin{enumerate}[(i)]
\item $A$ is finitely generated over $\Z{Q_{\chi}}$;
\item $A$ is generated by $\AA$ over $\Z{Q_{\chi}}$;
\item there exists a matrix $\left(\lambda_{(a, a_1)} \mid (a, a_1) \in \AA^2 \right)$ 
with entries in $\Z{Q_{\chi}}$ with the following two properties:
\begin{enumerate}[a)]
\item $a_{1} = \sum\nolimits_{a \in \AA} a \cdot \lambda_{a,a_1}  $ for every $a_{1}\in \AA$,
\item $v_{\chi} (\lambda_{a,a_1} ) > r$ for all $(a,a_1) \in \AA^2$.
\end{enumerate}
\end{enumerate}
\end{prp}

\begin{proof}
Assume first that statement (i) holds. 
There exists then a finite subset $\BB$ of $A$ 
so that every element $x \in A$ is a linear combination of the generators 
$b \in \BB$ with coefficients $\nu_b \in\Z{Q_\chi}$.
Since $\AA$ generates $A$ as a $\Z{Q}$-module, 
each generator $b \in \BB$ is a linear combination of $\AA$ with coefficients in $\Z{Q}$,
say $ b = \sum \nolimits _{a \in \AA} a \cdot \mu_{a,b} $.
It follows that 
\[
x =  
\sum \nolimits _{b}  b \cdot  \nu_b 
=
\sum \nolimits _{b} 
\left(\sum \nolimits _{a \in \AA} a \cdot \mu_{a,b}\right) \cdot  \nu_{b} 
=
\sum\nolimits _{a} a \cdot  \sum \nolimits _{b} \mu_{a, b}\nu_b .
\]
Lemma \ref{lem:Properties-valuation-vsubchi} 
allows us to estimate the $v_\chi$-value of the coefficient of the generator $a$ from below: 
\begin{align*}
v_\chi\left(\sum \nolimits _{b} \mu_{a, b}\nu_b \right) 
&\geq 
\min \{v_\chi(\mu_{a,b} \nu_b )\mid b \in \BB \}
\geq  
\min \{v_\chi(\mu_{a,b}) +  v_\chi( \nu_b) \mid b \in \BB \}\\
&\geq
\min \{v_\chi(\mu_{a,b})  \mid b \in \BB \} = m_a.
\end{align*}
This calculation shows that every element of $A$ is a linear combination of the generators $a$ 
with coefficients whose $v_\chi$-values are bounded from below by  
$m = \min \{m_a \mid a \in \AA \}$.
Choose now $q_0 \in Q$ with  $\chi(q_0 ) > r + \min\{0, -m \}$.
For each $a_1 \in \AA$ 
there exists then coefficients $\lambda_{a,a_1}$  with $v_\chi(\lambda_{a.a_1}) \geq m$
so that the equation
\[
a_1 \cdot q_0^{-1} = \sum\nolimits_{a \in \AA} a \cdot \lambda'_{a,a_1}
\]
holds. 
If we set $\lambda_{a,a_1} = \lambda'_{a,a_1} q_0$ 
we obtain a matrix that satisfies requirements a) and b) enunciated in (iii).
\smallskip

Assume next (iii) holds and $(\lambda_{a, a_1})$ is a matrix satisfying requirements a) and b).
We shall prove,  by induction on $m$,    
that $a_{1} \cdot q$ lies in the subgroup $ \sum_{a_1 \in \AA} a_1 \cdot \Z{Q_{\chi}}$ 
for every  $ q\in Q$ with $\chi(q) \geq -m \cdot r$.
If $m=0$ then  $q \in Q_{\chi}$ and so the claim holds;
if it holds for $m \geq 0$ and $\chi(q) \geq -(m+1) \cdot r$,
then
\[
a_1 \cdot q = \left( \sum\nolimits_{a \in \AA} a \cdot \lambda_{a,a_1} \right) \cdot q
= 
\sum\nolimits_{a \in \AA}  a \cdot (\lambda_{a,a_1}\cdot q)
\]
and
$v_{\chi}( \lambda_{a,a_1} \cdot q)= v_{\chi}(\lambda_{a,a_1}) + \chi(q) \geq  -m \cdot r$.
Going back to the definition of  $v_{\chi}$,
one sees that each element $q'$ occurring in the support of $\lambda_{a,a_1} \cdot q$ 
has $\chi$-value greater or equal to $-m \cdot r$ 
and lies thus in  the subgroup $\sum_{a_1 \in \AA} a_1 \cdot \Z{Q_{\chi}}$ by the induction hypothesis.
So statement (ii) is holds.
Since (ii) clearly entails statement (i), the theorem is established.
\end{proof}
\begin{remarks}
\label{remarka:Sigma0-criterion}
a) The above result goes back to \cite{Str81b};
it generalizes a similar proposition of an article by R. Bieri und R. Strebel,
namely \cite[Prop. 2.1]{BiSt80}.

b) The matrix $\left(\lambda_{(a, a_1)} \mid (a, a_1) \in \AA^2 \right)$ 
induces an endomorphism $\varphi_0$ of the free right $\Z{Q}$-module $F_0$ with basis $\AA $.
This endomorphism can be reinterpreted as a chain map lifting the identity of the obvious epimorphism
$\delta_0 \colon  F_0  \epi A$. 
In this form the criterion 
generalizes to a criterion for the invariants $\Sigma^m$ introduced by R. Bieri and B. Renz in \cite{BiRe88}
(see Theorem C).
\end{remarks}

\subsubsection{Openness of $\Sigma^0(Q;A)$}
\label{sssec:Sigma0-criterion-basic-result}
%
The $\Sigma^0$-criterion provides one with a cover of $\Sigma^0(Q;A)$ by open subsets
and thus implies that the subset $\Sigma^0(Q;A)$ of the sphere $S(Q)$ is open.
Indeed,
if $(\lambda_{a, a_1} \mid (a,a_1) \in \AA^2)$ is any matrix 
satisfying condition b) in statement (iii) with $r = 0$,
the number
\[
\min \{ v_{\chi}(\lambda_{a,a_1})   \mid (a,a_1) \in \AA^2 \}
=
\min \{ \chi(q) \mid q \in \supp(\lambda_{a,a_1}) \text{ and } (a,a_1) \in \AA^2 \}
\]
is positive.
As only finitely many group elements are involved in this minimum,
the set 
\begin{equation}
\label{eq:Definition-open-cover-Sigma0}
\OO(\AA, (\lambda_{a, a_1} ))= \{[\chi'] \in S(Q) 
\mid  \min \{ v_{\chi'}(\lambda_{a,a_1})   \mid (a,a_1) \in \AA^2 \} > 0 \}
\end{equation}
is a non-empty open subset of $S(Q)$.
This fact, in conjunction with  Proposition  \ref{prp:Sigma0-criterion},
then implies 
\begin{crl}
\label{crl:Openness-Sigma0(Q,A)}
Let $Q$ be a finitely generated group and $A$ a finitely generated right $\Z{Q}$-module.
Then $\Sigma^0(Q;A)$ is an open subset of $S(Q)$.
\end{crl}
%
%
%
%
\subsubsection{An alternate definition of $\Sigma^0$ for polycyclic groups}
\label{sssec:Alternate-definition-polycyclic-groups}
%
In the definition of $\Sigma^0$ one asks,
given a non-zero character $\chi$,
whether the finitely generated $Q$-module $A$ is finitely generated over the monoid $Q_\chi$.
The submonoid $Q_\chi$ is, in general, infinitely generated.
So the requirement that $A$ be finitely generated over a \emph{finitely generated} submonoid of $Q_\chi$
is typically a more severe condition than that used in the definition of $Q_\chi$.
For polycyclic groups, however, the two conditions are equivalent.
This is a consequence of
\begin{prp}
\label{prp:Existence-finitely-generated-subring}
Assume $Q'$, the derived group of $Q$, is finitely generated.
Then the following conditions are equivalent for every non-zero character $\chi$ 
and every finitely generated $\Z{Q}$-module $A$: 
\begin{enumerate}[(i)]
\item  
$A$ is finitely generated over the monoid ring $\Z{Q_{\chi}}$, 
\item
$A$ is finitely generated over the monoid ring $\Z{M}$ of a finitely generated submonoid 
$M \subseteq Q_{\chi}$. 
\end{enumerate}
\end{prp}

\begin{proof}
Set $N =\ker \chi$ and $\bar{Q} = Q/N$. 
Then $\bar{Q}$ is free abelian; 
so there exists a finite subset $\YY= \{y_{1}, \ldots, y_{f} \}$
which maps onto a basis of $\bar{Q}$; 
replacing, if need be, some elements $y_{j} \in \YY$ by their inverses, 
we may assume that $\YY \subset Q_{\chi}$. 
Set $y_{0} = y_{1} \cdots y_{f}$  and $r_0= \chi(y_{0})$. Then $r_{0}> 0.$

Assume the module $A$ is finitely generated over $\Z{Q_{\chi}}$, 
say by the finite set $\AA$. 
Then implication (i) $\Rightarrow$ (iii) of Proposition \ref{prp:Sigma0-criterion}
provides one with a matrix $(\lambda_{a,a_1})$ that satisfies the equation
\[
a_1 = \sum\nolimits_{a \in \AA} a \cdot \lambda_{a,a_1}
\]
for each $a_1 \in \AA$ 
and whose entries satisfy the inequalities $v_{\chi}(\lambda_{a,a_1}) > r_0 = \chi(y_0)$.
Let $\ZZ$ be the union of the supports of the $\card(\AA)^2$ entries of $(\lambda_{a,a_1})$
and let $M$ be the monoid generated by set
\[
N  \cup \YY \cup \{z  \cdot y_0^{-1} \mid z \in \ZZ\}.
\]
Notice that $M$ is contained in $Q_\chi$.
By hypothesis, the derived group $Q'$ of $Q$ is finitely generated;
hence so is $N = \ker \chi$ and the monoid $M$.
 Moreover,
as $N$ contains the commutator subgroup of $Q$,
the monoid $M$ is stable under conjugation by the elements of $Q$. 
 
Set  $A_{0}= \sum\nolimits_{a \in \AA} a \cdot \Z{M}$.
Then $A_{0}$ is  a finitely generated $\Z{M}$-module. 
To prove statement (ii), 
it suffices therefore to show that $A_{0}$ coincides with $A$. 
As the monoid $M$ contains the kernel $N$ of $\chi$ and a basis $\YY$ of $\bar{Q} =Q/N$ 
there exists, for every element $q \in Q$, an integer $\ell ≥0$ such that $ q  \cdot y_0^\ell\in M$.
So we need only show that $A_{0} \cdot y_0^{-1} \subseteq A_{0}$.

For every $\mu \in \Z{M}$ and $a_{1} \in \AA$ the following computation is valid:
\[
a_1 \cdot \mu \cdot y_{0}^{-1}  
=
a_{1} \cdot y_0^{-1} \cdot y_0 \mu y_0^{-1}
=
\left(\sum\nolimits_{a \in \AA} a \cdot \lambda_{a,a_1}\cdot  y_{0}^{-1}\right)   \cdot y_0 \mu y_0^{-1}\]
Since $\Z{M}$ is invariant under conjugation by $Q$ 
and as the elements $ \lambda_{a,a_1} \cdot y_{0}^{-1}$ belong to $\Z{M}$ by the construction of $M$,
the computation shows that  $A_{0}  \cdot y_{0}^{-1}\subseteq A_{0}$.
But if so,
$A$ coincides with $A_{0}$,
whence $A$ is generated by the finite set $\AA$  over the finitely generated submonoid $M$ of $Q_{\chi}$. 

The converse is clear.
\end{proof}

\begin{remark}
\label{remark:Origin-prp:Existence-finitely-generated-subring}
Proposition \ref{prp:Existence-finitely-generated-subring} 
goes back to \cite{Str81b}.
Its proof uses an idea employed in establishing Proposition  2.4(ii) in \cite{BNS}.
\end{remark}
\subsubsection{Some consequences for polycyclic groups}
\label{sssec:Consequences-polycyclic groups}
%
We are now able  to show 
that,  for every finitely generated module $A$ over a polycyclic group $Q$,
the invariant $\Sigma^0(Q;A)$ can be determined by looking at the invariants of cyclic submodules. 
The first step on the route to this result is
\begin{prp}
\label{prp:Invariant-submodule-polycyclic-group}
Let $Q$ be a polycyclic group 
and $A_{1}$ a submodule of the finitely generated $\Z{Q}$-module $A$. 
Then $A_{1}$ is finitely generated over $\Z{Q}$ and 
\begin{equation}
\label{eq:Invariant-submodule-polycyclic-group}
\Sigma^{0}(Q;A) \subseteq \Sigma^{0}(Q;A_{1}). 
\end{equation}
\end{prp}

\begin{proof}
The group ring of a polycyclic group is noetherian 
(by P. Hall’s extension of Hilbert’s Basis Theorem; see, \eg, \cite[15.3.3]{Rob96})
and so $A_{1}$, 
being a submodule of a finitely generated $\Z{Q}$-module, 
is finitely generated over $\Z{Q}$. 

Assume now that $\chi \colon  Q \to \R$ is a non-zero character
and that $A$ is finitely generated over the monoid ring $\Z{Q_{\chi}}$. 
As $Q$ is polycyclic
Proposition \ref{prp:Existence-finitely-generated-subring} applies and 
tells us that 
$A$ is finitely generated over the monoid ring  $\Z{M}$  of a finitely generated
submonoid  $M$ of $Q_{\chi}$
We may, and shall, assume that $M$ contains the derived group of $Q$ 
and can then apply another variant of Hilbert’s Basis Theorem 
(see Chapter 10,  p.\;423, Theorem 2.6 in \cite{Pas85}) 
and thus infer that $\Z{M}$ is a noetherian ring. 
The submodule $A_{1}$ is therefore finitely generated over $\Z{M}$, 
hence a fortiori over $\Z{Q_{\chi}}$ and  so $[\chi]$ belongs to $\Sigma^{0}(Q;A_1)$. 
\end{proof}

As a first consequence of Proposition  \ref{prp:Invariant-submodule-polycyclic-group},
we have the following refinement of Lemma \ref{lem:ses-modules}:
\begin{crl}
\label{crl:Invariants-ses-over-polycyclic-group}
Assume $Q$ is a polycyclic group. 
Then the invariants of the terms of a short exact sequence $A_{1}\mono A \epi A_{2}$
of finitely generated $\Z{Q}$-modules are related by the identity 
\begin{equation}
\label{eq:Invariants-ses-over-polycyclic-group}
\Sigma^{0}(Q;A) = \Sigma^0(Q;A_{1}) \cap \Sigma^0(Q;A_{2}). 
\end{equation}
\end{crl}

Proposition \ref{prp:Invariant-submodule-polycyclic-group} 
and Corollary \ref{crl:Invariants-ses-over-polycyclic-group}
together lead then to a $\Sigma^0$-criterion for polycyclic groups:
\begin{crl}
 \label{crl:Sigma0-criterion-modules-polycyclic-groups}
Let $Q$ be a polycyclic group, 
 $\chi \colon Q \to \R$ a non-zero character and $A$ a $\Z{Q}$-module 
 that is generated by the finite set $\AA$. 
 Then the following conditions are equivalent: 
\begin{enumerate}[(i)]
 \item $[\chi] \in \Sigma^0(Q;A)$, 
\item for every $a \in \AA$, there exists,  an element $\lambda_a \in Q_\chi$ so that 
$a = a\cdot \lambda_a$ and $v_{\chi}(\lambda_a ) > 0$.
\end{enumerate}
   \end{crl}
    
\begin{proof}
 If $[\chi] \in \Sigma^0(Q;A)$, 
 Proposition \ref{prp:Invariant-submodule-polycyclic-group}
 implies that $[\chi]$ lies in the invariant of every cyclic submodule  $A_a = a \cdot \Z{Q}$ with $a \in \AA$.
Proposition \ref{prp:Sigma0-criterion}
next provides one, for each generator $a \in \AA$, with an element $\lambda_a$ having the stated properties. 
So hypothesis (i) implies assertion (ii); 
the converse is covered by Proposition \ref{prp:Sigma0-criterion}. 
\end{proof}
 %
%
%
\section{Tame modules}
\label{sec:Tame-modules}
%
%
In the next section, 
finitely presented groups $G$ will be constructed 
that are semidirect products $Q \ltimes A$ 
of an abelian normal subgroup $A$ by a nilpotent group $Q$.
The module $A$ will be assumed to be \emph{tame}. 
In this section, tame modules are defined and discussed.
\begin{definition}
\label{definition:Tame-module}
Let $Q$ be a finitely generated \emph{nilpotent} group 
and $A$ a finitely generated \emph{right} module over $\Z{Q}$.
We call $A$ \emph{tame} 
if there exists a central series $\{1\} = Q_{k+1} < Q_k < \cdots < Q_2 < Q_1 = Q$ of $Q$ 
and a finite  $\Z{Q}$-generating set $\AA \subset A$ such that
\begin{equation}
\label{eq:Tame-module}
\Sigma^0(Q_i;\AA \cdot \Z{Q_i}) \cup  -\Sigma^0(Q_i;\AA \cdot \Z{Q_i}) \supseteq S(Q_i, Q_{i+1}) 
\text{ for } 1 \leq i \leq k.
\end{equation}
\end{definition}
%
\subsection{Dependence on the generating set}
\label{ssec:Dependence-generating-set}
%
We first show that a $\Z{Q}$-module 
which is tame with respect to a central series $\{Q_i\}$ 
and some finite $\Z{Q}$-generating set $\AA \subset A$,
is tame with respect to $\{Q_i\}$ and any other finite generating set $\BB \subset A$.

\begin{prp}
\label{prp:Independence-generating-set}
Let $Q$ be a finitely generated nilpotent group and $\{Q_i \}$ a central series of $Q$.
If $A$ is a finitely generated $\Z{Q}$-module and $\AA$, $\BB$ are two finite $\Z{Q}$-generating sets of $A$ 
then
\begin{equation}
\label{eq:Independence-generating-system}
\Sigma^0(Q_i; \AA \cdot \Z{Q_i}) \cap S(Q_i,Q_{i+1}) = \Sigma^0(Q_i; \BB \cdot \Z{Q_i})\cap S(Q_i,Q_{i+1})
\end{equation}
for $ i \in \{1,2, \ldots, k\}$.
\end{prp}

\begin{proof}
We start with two easy reductions.
Since $\AA \cup \BB$ is also a finite $\Z{Q}$-generating set we can assume that $\AA \subset \BB$;
as $\BB$ is finite, 
it suffices to treat the case where $\BB = \AA \cup \{b\}$.

As $\AA$ generates the module $A$ the element $b$ of $A$ can be written as 
$b = \sum\nolimits_{a \in \AA} a \cdot \lambda_a$
where each $\lambda_a$ is an element of $\Z{Q}$.
Let $E$ denote the union of the supports of all the elements $\lambda_a$ and of  the singleton $\{1_Q\}$.
Then the chain of inclusions
\[
\AA \cdot \Z{Q_i}
 \subseteq \BB \cdot \Z{Q_i} 
= (\AA \cup \{b\}) \cdot \Z{Q_i} 
 \subseteq \AA{E}\cdot \Z{Q_i}
\]
holds for every $i \in \{1, \ldots, k\}$.
By Proposition \ref{prp:Invariant-submodule-polycyclic-group}
this chain of inclusions gives rise to to a chain of Sigma-invariants,
namely
\begin{equation}
\label{eq:Descending-chain-of-Sigmas}
\Sigma^0(Q_i;\AA \cdot \Z{Q_i})  \supseteq \Sigma^0(Q_i;\BB \cdot \Z{Q_i})
 \supseteq \Sigma^0(Q_i;\AA{E}\cdot \Z{Q_i}).
\end{equation}
Now the submodule $ \AA{E}\cdot \Z{Q_i}$ will be finitely generated over a subring $R$ of $\Z{Q_i}$ 
if each summand $\AA q \cdot \Z{Q_i}$ with $q \in E$  is finitely generated over  $R$,
whence the inclusion
\begin{equation}
\label{eq:Passing-to-individual-elements}
\Sigma^0(Q_i;\AA{E}\cdot \Z{Q_i}) \supseteq \bigcap\nolimits_{q \in E} \Sigma^0(Q_i;\AA q \cdot \Z{Q_i}).
\end{equation}
In view of this inclusion and the chain \eqref{eq:Descending-chain-of-Sigmas}
it suffices therefore to show 
that the intersection $\Sigma^0(Q_i;\AA q \cdot \Z{Q_i}) \cap  S(Q_i,Q_{i+1})$ does not depend on $q$.

Consider a point $[\chi] \in S(Q_i, Q_{i+1})$.
By Proposition \ref{prp:Invariant-submodule-polycyclic-group},
this point lies in $\Sigma^0(Q_i;\AA q \cdot \Z{Q_i})$ if, and only if, 
$\AA q \cdot \Z{Q_i} = \AA q \cdot \Z({Q_{i \chi}})$.
Now $Q_i$ is normal in $Q$ and so $\AA q \cdot \Z{Q_i}= (\AA \cdot \Z{Q_i})q$;
as $Q_i$ is central in $Q$ \emph{modulo} $Q_{i+1}$ and as $\chi$ vanishes on $Q_{i+1}$,
one has also $Q_{i \chi} = q Q_{i \chi} q^{-1}$,
and thus
\[
\AA q \cdot \Z({Q_{i \chi}}) = \AA \cdot q (\Z{Q_{i \chi}}) q^{-1} \cdot q 
= \AA \cdot   \Z{Q_{i \chi}} \cdot q. 
\]
So $\AA q \cdot \Z{Q_i} = \AA q \cdot \Z({Q_{i \chi}})$ holds precisely if
$(\AA  \cdot \Z{Q_i}) q = (\AA \cdot \Z{Q_{i \chi}}) q$.
\end{proof}

\begin{remarks}
\label{remark:Independence-generating-set}
a) The previous proposition implies that 
a finitely generated $\Z{Q}$-module $A$ which is tame with respect to a central series  of $Q$ 
and some finite generating set,
is tame with respect to this series and any other finite generating set.
In the sequel we shall call a finitely generated $\Z{Q}$-module $A$ 
\emph{tame with respect to a central series $\{Q_i \}$ of $Q$},
if it is tame with respect to $\{Q_i \}$ and some finite generating set.

b) I have not been able to determine how the tameness of a module $A$ depends on the central series $\{Q_i \}$. 
One might hope that a module which is tame with respect to some central series $\{Q_i \}$
is also tame with respect to some canonical central series,  for instance with respect to the upper central series. 
(The lower central series does not qualify; 
indeed, the group $G_2$ discussed in Theorem 2 of \cite{RoSt82} shows 
that a module which is tame with respect to the upper central series 
need not be tame with respect to the lower central series.)

The construction in section \ref{sec:Construction-groups} will therefore be carried out for an arbitrary central series;
fortunately, the construction for a general central series is no more involved
than that for an explicitly known central series.
\end{remarks}

\subsection{Submodules and quotient modules}
\label{ssec:Submodules-tame-modules}

%
%
Proposition \ref{prp:Independence-generating-set}
has many useful consequences. 
A first one is the fact that submodules and quotient modules of tame modules are tame:
\begin{crl}
\label{crl:Tameness-sub-and-quotient-modules}
Let $Q$ be a finitely generated nilpotent group,
$\{Q_i \}$ a central series of $Q$ and $A$ is a finitely generated $\Z{Q}$-module.
Suppose $A$ is tame with respect to $\{Q_i \}$.
Then all $\Z{Q}$-submodules $A_1 \leq A$, all homomorphic images $A_2$ of $A$ and all direct sums
$A \oplus \cdots \oplus A$ of finitely many copies of $A$ are tame with respect to $\{Q_i \}$.
\end{crl}

\begin{proof}
Let $\AA$ be a finite generating set of $A$.
Since the ring $\Z{Q}$ is noetherian, 
every submodule $A_1$ of $A$ is finitely generated over $\Z{Q}$,
say by the finite subset $\BB$.
By Proposition \ref{prp:Independence-generating-set}, 
the module $A$ is then tame with respect to the central series $\{Q_i \}$ and $\AA \cup \BB$.
As $\BB \cdot \Z{Q_i}$ is a $\Z{Q_i}$-submodule of $(\AA \cup \BB) \cdot \Z{Q_i}$,
Proposition \ref{prp:Invariant-submodule-polycyclic-group}
applies and yields the inclusion 
$\Sigma^0(Q_i; (\AA \cup \BB) \cdot \Z{Q_i}) \subseteq \Sigma^0(Q_i;  \BB \cdot \Z{Q_i})$.
So 
\begin{align*}
&\phantom{\supseteq \hspace*{2mm}}\Sigma^0(Q_i; \BB \cdot \Z{Q_i}) 
\cup -\Sigma^0(Q_i; \BB \cdot \Z{Q_i}) \\
&\supseteq 
\Sigma^0(Q_i; (\AA \cup \BB) \cdot \Z{Q_i}) \cup -\Sigma^0(Q_i; (\AA \cup \BB) \cdot \Z{Q_i})
\supseteq S(Q, Q_{i+1})
\end{align*}
for every $i$ whence $A_1$ is tame with respect to the central series $\{Q_i \}$.

Consider next a homomorphic image $A_2$ of $A$.
Then $A_2$ is generated by the image $\AA_2$ of $\AA$ in $A_2$ 
and so it is tame with respect to $\{Q_i \}$ and $\AA_2$; 
indeed, 
$\Sigma^0(Q_i; \AA_2 \cdot \Z{Q_i}) \supseteq \Sigma^0(Q_i; \AA \cdot \Z{Q_i}) $
by Lemma \ref{lem:ses-modules}.

Consider, thirdly,  a direct sum $A_1 \oplus \cdots \oplus A_m$ of copies $A_j$ of $A$ 
and let $\AA$ up to $\AA_m$ denote the corresponding copies of the finite generating set $\AA$ of $A$.
The definition of $\Sigma^0$ immediately implies 
that $\Sigma^0(Q_i; \AA \cdot \Z{Q_i})$ is contained in
$\Sigma^0(Q_i; (\AA \cup \cdots \cup  \AA_m) \cdot \Z{Q_i})$,
whence $A_1 \oplus \cdots \oplus A_m$  is tame with respect to the central series $\{Q_i \}$.
\end{proof}
%
\subsection{Passage to subgroups of finite index}
\label{ssec:Passage-sub-groups-finite-index}
%
In the construction to be carried out in section \ref{sec:Construction-groups},
it is convenient to assume that the factors of the central series $\{Q_i \}$ are torsion-free.
The next result shows that this assumption is harmless.
\begin{prp}
\label{prp:Passage-subgroup-finite-index}
Let $Q$ be a finitely generated nilpotent group, $\{Q_i \}$ a central series of $Q$
and $A$ a finitely generated $\Z{Q}$-module.
\begin{enumerate}[(i)]
\item There exists a subgroup of finite index $P$ of $Q$ such that the induced central series 
$\{P_i = P \cap Q_i \}$ has torsion-free factors.
\item
If $P \leq Q$ is a subgroup of finite index,
then $A$ is tame with respect to $\{Q_i \}$ if, and only if, 
$A$, viewed as a $\Z{P}$-module, is tame with respect to the induced central series $ \{P_i = P \cap Q_i \}$.
\end{enumerate}
\end{prp}

\begin{proof}
(i) The proof is by induction on $k$, the number of factors of the central series $\{Q_i \}$.
By a result of K. Hirsch, every polycyclic group, hence every finitely generated nilpotent group, 
is residually finite (\cite{Hir54}, cf.\  \cite[1.3.10]{LeRo04}).
This property allows one to find a subgroup $Q' \leq Q$ of finite index 
which avoids the torsion-subgroup of $Q_k$.
Set $Q'_i = Q' \cap Q_i$.
By the induction hypothesis,
the group $\overline{Q'} = Q'/Q'_k$ has then a subgroup $\bar{P}$ of finite index 
such that the induced series  $\{\bar{P}_i = (\bar{P} \cap (Q'_i/Q'_k) \}$ has torsion-free factors.
Let $P$ be the preimage of  $\bar{P}$ in $Q'$.
Then $P$ is a subgroup of finite index in $Q$ 
and all the factors of the central series $\{P_i = P \cap Q_i \}$ 
are torsion-free. 

(ii) Since  $P$ has finite index in $Q$, 
the module $A$ is finitely generated over $P$; 
let $\AA$ be a finite $\Z{P}$-generating set.
Consider now an index $i \leq k$.
Since $P_i = P\, \cap \,Q_i$ has finite index in $Q_i$
and as $P_i  \cap [Q_i, Q_i] \subseteq P \cap Q_{i+1} = P_{i+1}$,
Lemma \ref{lem:Subgroup-finite-index-Sigma0}
allows one to deduce 
that the inclusion $\mu_i \colon P_i \incl Q_i$
 induces an isomorphism of spheres $\mu^*_i \colon S(Q_i, Q_{i+1}) \iso  S(P_i, P_{i+1})$ with
\begin{equation}
\label{eq:Relation-invariants}
\mu_i^* \left(\Sigma^0(Q_i; \AA \cdot \Z{Q_i}) \cap S(Q_i, Q_{i+1})\right)
=
\Sigma^0(P_i; \AA \cdot \Z{Q_i}) \cap S(P_i, P_{i+1}).
\end{equation}
Now $\AA \cdot \Z{P_i}$ is a submodule of $\AA \cdot \Z{Q_i}$
and $\AA \cdot \Z{Q_i}$ is a quotient module of a finite direct sum of $\AA \cdot \Z{Q_i}$; 
so $\Sigma^0(P_i; \AA \cdot \Z{P_i})$ coincides with $\Sigma^0(P_i; \AA \cdot \Z{Q_i})$
by Proposition \ref{prp:Invariant-submodule-polycyclic-group}  and Lemma \ref{lem:ses-modules}.
 Relation \eqref{eq:Relation-invariants} therefore implies
that $A$ is tame with respect to $\{Q_i\}$ if, and only if, 
it is tame with respect to the induced central series $\{P_i \}$.
\end{proof}
%
%
\section{Construction of finitely related groups}
\label{sec:Construction-groups}
%
%
In this section we establish our main result:
\begin{thm}
\label{thm:Construction-fp-group}
Let $Q$ be a nilpotent group and $A$ a finitely generated $\Z{Q}$-module.
If $A$ is tame the semidirect product $G = Q \ltimes A$ of $A$ by $Q$ has a finite presentation.
\end{thm}

The proof of this theorem will be carried out in several steps:
we start out with a reduction, 
then specify a finite generating set and an infinite list of defining relations
of a subgroup $G_1 = P \ltimes A$ with finite index in $G$.
Next the geometric assumption on $A$ will be used to define radii $\rho_i$.
In the fourth step, these radii go into the construction of a finitely related group $\tilde{G}$.
This group $\tilde{G}$ is a split extension of a group $\tilde{A}$ by the nilpotent group $P$
and admits of an epimorphism 
\[
\kappa \colon \tilde{G} = P \ltimes \tilde{A} \epi G_1 = P \ltimes A.
\]
In the last step, 
geometric arguments will be invoked to deduce 
that $\kappa $ is an isomorphism.

%
\subsection{First steps}
\label{ssec:First-steps}
%
Let $Q$ be a nilpotent group and $A$ a finitely generated (right) $\Z{Q}$-module,
that is tame with respect to a central series $\{Q_i \}$ of $Q$.
In this section, we carry out two easy steps of the proof of Theorem 
\ref{thm:Construction-fp-group}.
%
\subsubsection{Step 1: reduction}
\label{sssec:Step1-reduction}
%
By claim (i) of Proposition \ref{prp:Passage-subgroup-finite-index}
there exists a subgroup $P \leq Q$ of finite index
such the induced central series $\{P_i = P\cap Q_i \}$ has torsion-free, hence free abelian, factors.
Moreover, 
by claim (ii) of the proposition
the $\Z{Q}$-module $A$ remains finitely generated when viewed as a $\Z{P}$ module
and it is tame with respect to the central series $\{P_i\}$.
Set $G_1 = P \ltimes A$.

Since $G$ admits a finite presentation if, and only if, $G_1$ has this property,
there is no harm in assuming from the very outset 
that the factors $Q_i/Q_{i+1}$ of the given central series
\begin{equation}
\label{eq:definition-central-series}
Q_{k+1} = \{1\} < Q_k < \cdots < Q_2 < Q_1 = Q
\end{equation}
are free-abelian, say of rank $n_i$.
%
\subsubsection{Step 2: choice of an infinite presentation}
\label{sssec:Step2-infinite-presentation}
%
Construct a finite set of generators
$
\XX = \TT_k \cup \TT_{k-1} \cup \cdots  \cup\TT_2 \cup \TT_1, 
$
of the nilpotent group $Q$ by picking, for each index $i \in \{1,2, \ldots, k \}$, 
a finite subset
\[
\TT_i = \{t_{1,i}, t_{2,i}, \ldots , t_{n_i, i} \}
\]
of $Q$ 
whose image under the canonical map $\can \colon Q  \epi Q/Q_{i+1}$ 
is a basis of the free abelian group $Q_i/Q_{i+1}$.
Enlarge $\XX$ to a finite generating set $G = Q \ltimes A$ by adding to $\XX$ 
a finite set $\AA$ of $\Z{Q}$-generators of $A$. 
The union $\AA \cup \XX$ is then finite and generates $G$.

Now to an infinite presentation of $G$.
Let $F$, $F(\AA)$ and $F(\AA \cup \XX)$ 
denote the free groups on the finite sets $\XX$, $\AA$, and $\AA \cup \XX$, respectively.
The group $Q$ is a finitely generated nilpotent group and so it is finitely related by a result of P. Hall's
(cf. \cite[2.2.4]{Rob96}).
There exists therefore a finite set $\RR_\Q$ of reduced words in $\XX$
so that the epimorphism $\pi$ induced by the inclusion $\XX \incl Q$ 
gives rise to an isomorphism $\pi_* \colon F / \gp_F(\RR_Q) \iso Q$.

Next, define an infinite set of commutators of $G$, namely the set
\begin{equation}
\label{eq:Definition-KK}
\KK = \{ [a,b^w] \mid (a,b) \in \AA^2 \text{ and } w \in F \}.
\end{equation}
Let $N$ be the \emph{normal} closure of the finite set $\AA$ in the free group $F(\AA \cup \XX)$
and let $\bar{N}$ denote the canonical image of $N$ in the factor group 
\begin{equation}
\label{eq:Definition-auxiliary-quotient}
\langle \AA \cup \XX \mid \KK\rangle.
\end{equation}
We claim that every element $\bar{a} \in \bar{\AA}$ commutes 
with every conjugate $\bar{b}^w$ of every $\bar{b} \in \bar{\AA}$;
here $w$ is a freely reduced word in $F(\AA \cup \XX)$.
By definition, this claim holds if $w \in F = F(\XX)$.
Suppose now that $w \in F(\AA \cup \XX) \smallsetminus F$
 and that $c \in \AA^{\pm}$ is the first letter in $w$ that lies outside $\XX^{\pm}$; 
 say $w = u \cdot c \cdot v$  with $u \in F$.
 Then the following computation is valid (congruences \emph{modulo} $\KK$):
 \begin{align*}
 [a,b^w] = [a,b^{u c v}] = [a^{v^{-1}}, b^{uc}]^v&=  [a^{v^{-1}}, b^u \cdot [b^u, c]]^{v}
 \equiv [a^{v^{-1}}, b^u]^{v} = [a,b^{uv}].
 \end{align*}
 The claim thus follows by descending induction on the letters of $w$ outside $\XX^{\pm}$.
 
 Consider next the quotient $ H= \langle \AA \cup \XX \mid \KK \cup \RR_Q \rangle$
of the free group $F(\AA \cup \XX)$; 
let $\hat{A}$ denote the canonical image of $N = \gp_{F(\AA \cup \XX)}(\AA)$ in $H$.
Then $\hat{A}$ is abelian and conjugation by $Q$ turns it into a (right) $\Z{Q}$-module.
As this module is finitely generated,
it satisfies the ascending chain condition on submodules  by a generalization of Hilbert's basis theorem 
\footnote{\cite[p.\,429, Theorem 1]{Hal54}, cf. \cite[p.\;464, 15.3.3]{Rob96}.}.
So the kernel of the canonical projection $\hat{A} \epi A$ 
is generated by a finite subset, say $\widehat{\RR_A}$, over the noetherian ring $\Z{Q}$. 
Lift $\widehat{\RR}_A$ to a finite subset  $\RR_A$ of $N = \gp_{F(\AA \cup \XX)} (\AA)$.
Then the inclusion of $\AA \cup \XX$  into $G = Q \ltimes A$ induces an isomorphism
\begin{equation}
\label{eq:Infinite-presentation-G}
\iota_* \colon \langle \AA \cup \XX \mid \RR_A \cup \KK \cup \RR_Q) \iso G.
\end{equation}
This is the infinite presentation announced in the heading of section \ref{sssec:Step2-infinite-presentation}.

We are left with the task of showing 
that the infinite set of commutators $\KK$ in presentation \eqref{eq:Infinite-presentation-G} 
can be deduced from finitely many among them and the finite set  $\RR_A \cup \RR_Q$.
We shall reach this goal on a roundabout route:
we shall specify a finite subset $\KK_0$ of $\KK$,
but supplement it by a finite set $\CC$, 
which paraphrases the \emph{hypothesis  that $A$ be a tame $\Z{Q}$-module}.
The set $\KK$ will then be deduced from the finite set of relators 
$\RR_A \cup \KK_0 \cup \CC \cup \RR_Q$.
%
\subsection{Step 3: finding the radii $\rho_i$}
\label{ssec:Step3-radii}
%
Now we bring into play the geometric assumption 
that $A$ is a tame $\Z{Q}$-module with respect to the central series $\{Q_i \}$;
 by step 1) the factors of this central series can be assumed to be free abelian.
The module $A$ is tame with respect to some finite $\Z{Q}$-generating set, say $\BB \subset A$;
by remark \ref{remark:Independence-generating-set}
it is therefore tame with respect to the previously chosen set $\AA \subset A$.
For each index $i \in \{1, 2, \ldots, k\}$,
the inclusion
\[
S(Q_i, Q_{i+1}) \subseteq \Sigma^0(Q_i; \AA \cdot \Z{Q_i}) \cup - \Sigma^0(Q_i; \AA \cdot \Z{Q_i})
\]
thus holds. 
Our next aim is to define radii $\rho_i$, one for each index $i \in \{1, 2, \ldots, k\}$.
These radii will be found independently of each other.
In the remainder of step 3, the index $i \in \{1, 2, \ldots, k\}$ will therefore be fixed.
%
\subsubsection{Construction of finite open covers}
\label{sssec:Construction-finite-open-covers}
%
As detailed in section \ref{sssec:Sigma0-criterion-basic-result},
the tameness of $A$ allows one to construct open covers $\{\OO_{\Lambda(i)}\}$  
of the invariants $\Sigma^0(Q_i; \AA \cdot \Z{Q_i})$.
In view of corollary \ref{crl:Sigma0-criterion-modules-polycyclic-groups},
the matrices involved in this cover can and will be assumed to be diagonal.
The collection $\{\OO_{\Lambda(i)}\}$ gives then rise to an open cover 
$\{\OO_\Delta(i)\} \cup \{-\OO_\Delta(i)\}$
of 
\[
\Sigma^0(Q_i; \AA \cdot \Z{Q_i}) \cup - \Sigma^0(Q_i; \AA \cdot \Z{Q_i})
\]
and hence of the subsphere $S(Q_i, Q_{i+1})$.
As this subsphere is compact, 
there exist finitely many diagonal matrices
\[
\Delta_{1,i} = \{\delta_{a, a_1} \cdot \lambda(a;1,i) \}, \;
\Delta_{2,i} = \{\delta_{a, a_1}  \cdot \lambda(a;2,i) \}, \ldots,
\Delta_{\ell_i,i} = \{\delta_{a, a_1} \cdot \lambda(a;\ell_i,i) \}
\]
whose associated open sets $\OO_{\Delta_{1,i}}$, $\OO_{\Delta_{2,i}}$,\ldots, $\OO_{\Delta_{\ell_i,i}}$
cover, \emph{together with their antipodal images $-\OO_{\Delta_{j, i}}$},
the entire sphere $S(Q_i, Q_{i+1})$.

The diagonal matrices $\Delta_{j,i}$ have entries in $\Z{Q_i}$ and enjoy two properties:
\begin{align}
&a = a \cdot \lambda(a; j, i) \text{ for every } a \in \AA, \text{ and }
\label{eq:Centraliser-property}\\
&v_\chi(\lambda(a; j, i)) > 0 \text{ for every } a \in \AA \text{ and every } [\chi] \in \OO_{\Delta_{j, i}}.
\end{align}
%
\subsubsection{Passage to euclidean spaces}
\label{sssec:Passage-euclidean-spaces}
%
As before, 
let $n_i$ be the rank of the free-abelian group $Q_i/Q_{i+1}$.
Next, let $\vartheta_i$ be the epimorphism
\begin{equation}
\label{eq:Epimorphisms-theta-i}
\vartheta_i \colon Q_i \epi Q_i/Q_{i+1} \iso \Z^{n_i}
\end{equation}
which sends the generator $t_{j,i} \in \TT_i$ 
onto the $j$-th standard basis vector of  the euclidean lattice $\Z^{n_i}$.
This epimorphism induces an isomorphism
\begin{equation}
\label{eq:Isomorphism-theta-i-induced}
\vartheta_i^* \colon \s^{n_i - 1} \iso S(Q_i, Q_{i+1})
\end{equation}
of the unit sphere in the $n_i$-dimensional euclidean space $\R^{n_i}$ onto the sphere $S(Q_i, Q_{i+1})$.
It takes a unit vector $u$ to the ray $[\chi_u]$ represented by the character $\chi_u$, 
the character that sends the element $q \in Q_i$ to the scalar product $\langle u,\vartheta_i(q) \rangle$.

We next introduce finite subsets $L_{j,i}$ of the standard lattice $\Z^{n_i} \subset \R^{n_i}$;
they are defined by
\begin{equation}
\label{eq:Definition-L-sub-ji} 
L_{j,i} = \bigcup\nolimits_{a \in \AA}\vartheta_i(\supp \lambda(a;j,i)).
\end{equation}
The definitions of the homeomorphism $\theta_i^*$ and of the sets $\OO_{\Delta_{j,i}}$,
together with the covering property of the family $\{\OO_{\Delta_{j,i}} \} \cup \{-\OO_{\Delta_{j,i}} \}$,
imply then the following geometric property of the family $\FF_{+, i} =\{L_{j,i} \mid 1 \leq j \leq \ell_i \}$
of finite subsets of the lattice $\Z^{n_i}$:
\emph{for each unit vector $u \in \s^{n_i	- 1}$ there exists an index $j = j(u)$ 
such that}
\begin{equation}
\label{eq:Property-subsets-L-sub-ji}
\langle u, L_{j, i}\rangle > 0 \quad \text{or} \quad \langle u, L_{j, i} \rangle< 0.
\end{equation} 
In the above, a finite subset of real numbers is considered to be positive, respectively negative,
if all its elements are positive, respectively negative.

The stated property of $\FF_{+,i}$ can be expressed more simply in terms of the auxiliary collection 
\[
\FF_i = \{ L \subset \Z^n_i \mid L \in \FF_{+,i} \text{ or } -L \in \FF_{+,i} \}.
\]
The new finite family fulfills the hypotheses of the following

\begin{lem}[{\cite{Str84}}, p.\;291, Lemma 25]
\label{lem:Geometric-lemma}
Let $\FF$ be a finite collection of finite subsets $L$ of the standard lattice $\Z^n \subset \R^n$.
Assume that for every $u$ in the sphere $\s^{n-1}$ 
there exists $L \in \FF$ such that $\langle u, L \rangle > 0$.
Then there exists a natural number $p_0$ 
such that for every lattice point $x$ with $\langle x, x \rangle = p+1 > p_0$ 
there exists $L \in \FF$ such that $x + L$ is contained in the ball 
$\B_p = \{y \in \Z^n \mid \norm{y}^2 \leq p \}$.
\end{lem}
We apply this lemma to the collection $\FF_i$, obtain an integer $p_0(i)$ 
and then set $\rho_i = \sqrt{p_0(i)}$. 
By varying $i$ from 1 to $k$ we then obtain the radii mentioned in the heading of step 3.
%
\subsection{Step 4: introduction of the group  $\tilde{G}$}
\label{ssec:Introduction-group-tildeG}
%
Recall that $F$ is the free group with basis $\XX = \TT_1 \cup \cdots \cup \TT_k$; 
see section \ref{sssec:Step2-infinite-presentation}.
In the sequel, we shall need certain free factors $F_i$ of $F$;
here $i$ ranges from $1$ to $k$.
By definition, the group $F_i$ has the subset
\[
\XX_i = \TT_i \cup \TT_{i+1} \cup\cdots \cup \TT_k
\]
as its basis. 
Next we need the subset $F^o$ of \emph{ordered} words of $F$.
It consists of all words $w$ of the form 
\begin{equation}
\label{eq:Definition-ordered-words}
w = t_{1,1}^{m_{1,1}} t_{2,1}^{m_{2,1}} \cdots t_{n_1,1}^{m_{n_1,1}}
\cdot
 t_{1,2}^{m_{1,2}} \cdots t_{n_2,2}^{m_{n_2,2}}
\cdots
t_{1,k}^{m_{1,k}} t_{2,k}^{m_{2,k}}\cdots t_{n_k,k}^{m_{n_k,k}}
\end{equation}
Finally, set $F^o_i =F_i \cap F^o$.
Note that the restriction $\pi{|F^0}$ of the projection $\pi \colon F \epi Q$ is bijective.
%
\subsubsection{Auxiliary relations of $\tilde{G}$}
\label{sssec:Auxiliary-relations}
%
As explained at the end of section \ref{sssec:Step2-infinite-presentation},
the infinite set of commutators $\KK$ in the presentation \eqref{eq:Infinite-presentation-G} of $G$ 
will be replaced by the union $\KK_0 \cup \CC$ of two finite subsets.
We are now ready to define them.

The set $\KK$ is made up of all commutators $[a, b^w]$ with $(a,b) \in \AA^2$ and $w \in F$. 
The subset $\KK_0$ is obtained from $\KK$ by restricting the conjugating words $w$ to a finite subset $W$. 
This subset is the complex product  
\begin{align}
W &= V_1 \cdot V_2 \cdots V_k \quad \text{where}
\label{eq:Definition-W}\\
V_i &= \{t_{1,i}^{m_{1,i}} t_{2,i}^{m_{2,i}} \cdots t_{n_i,i}^{m_{n_i,i}} \mid 
m^2_{1,i} + m^2_{2,i} + \cdots + m^2_{n_i,i} \leq \rho_i^2 \}.
\label{eq:Definition-V-sub-i}
\end{align}

Now to the set of relations $\CC$. 
These relations mimic the module relations \eqref{eq:Centraliser-property},
namely
$a = a \cdot \lambda(a; j, i)$.
We  rewrite them as relations of the form
\begin{equation}
\label{eq:Centraliser-property-bis}
a = \prod\nolimits_{u \in F^o_i} \left(a^{\lambda(a;j,i,\hat{u})}\right)^u.
\end{equation}
Here $\lambda(a;j,i;\hat{u})$ denotes the coefficient of the element $\lambda(a;j,i)$ at 
$q = \hat{u} = \pi(u)$.
The formally infinite product on the right hand side of relation  \eqref{eq:Centraliser-property-bis}  
is to be interpreted as a finite product, 
consisting of the factors with non-zero exponent $\lambda(a;j,i;\hat{u})$,
the factors being taken in an arbitrarily chosen order.
Note that $\CC$ comprises $\card(\AA) \cdot (\ell_1 + \ell_2 + \cdots + \ell_k)$ relations.
(The integers $\ell_i$ are defined in section \ref{sssec:Construction-finite-open-covers}.)
%
\subsubsection{Definition of $\tilde{G}$}
\label{sssec:Definition-tildeG}
%
At long last, we are able to define the group $\tilde{G}$.
This group is by construction finitely related;
it will later be shown to be isomorphic to $G$,
whence $G$ admits a finite presentation,
as claimed by Theorem \ref{thm:Construction-fp-group}.

The generating set of $\tilde{G}$ is the finite set 
\[
\AA \cup \XX \text{ with } \XX = \TT_1 \cup \cdots \cup \TT_i \cup \cdots \cup \TT_k \text{ and } 
\TT_i = \{t_{1,i}, t_{2,i}, \ldots, t_{n_i,i} \}.
\]
The set of defining relations is the union $\RR_A \cup \KK_0 \cup \CC \cup \RR_Q$.
The set of relators $\RR_A$ corresponds to a finite set of relations of the $\Z{Q}$-module $A$ 
with respect to the set of generators $\AA$;
the finite set of relators $\RR_Q$ defines the nilpotent group $Q$ as a quotient of the free group $F$ on $\XX$
(see section \ref{sssec:Step2-infinite-presentation}).
The finite sets $\KK_0$ and $\CC$ are as explained in the previous section \ref{sssec:Auxiliary-relations}.
The inclusions of $\AA$ into $A$ and of $\XX$ into $Q$ induce, 
by the choice of the set of relations $\RR_A \cup \KK_0 \cup \CC \cup \RR_Q$,  
an epimorphism $\kappa \colon \tilde{G} \epi G$.
Actually more is true:
\begin{prp}
\label{prp:Kappa-is-iso}
The epimorphism 
\[
\kappa \colon \tilde{G} = \langle \AA \cup \XX \mid  \RR_A \cup \KK_0 \cup \CC \cup \RR_Q\rangle
\epi 
G = Q \ltimes A
\] is an isomorphism.
It establishes that $G$ is finitely related.
\end{prp}
%
\subsection{Step 5: proof of Proposition \protect{\ref{prp:Kappa-is-iso}}}
\label{ssec:Step5-proof-proposition-4.3}
The proof will be by descending induction on $k$.
With this aim in mind,
we consider statements $\SS_i$, the index $i$ descending from $k+1$ to 1:
\begin{equation*}
\label{eq:Statement-SS-sub-i}
\SS_i \colon \left\{%
\begin{minipage}[c]{9.6cm}
\emph{the relation  $[a,b^w] = 1$ holds in $\tilde{G}$ for every couple $(a,b) \in \AA^2$\\ 
and every word $w = v_1 v_2 \cdots v_{i-1} w_i$ with\\ $v_1 \in V_1$, \ldots, $v_{i-1} \in V_{i-1}$
and $w_i \in F_i = F(\TT_i \cup \cdots \cup \TT_k)$.}
\end{minipage}
 \right \}
\end{equation*}
Statement $\SS_{k+1}$ holds because the set of conjugating exponents in $\KK_0$ 
is, by definition, $W = V_1 V_2 \cdots V_k$; this fact allows one to start the induction.
%
\subsubsection{Inductive step}
\label{sssec:Inductive-step}
%
Assume statement $\SS_{i+1}$ is valid for some index $i \leq k$.
We want to deduce, by an auxiliary induction,  that statement $\SS_{i}$ holds.
To this end,
we consider, for every natural number $p$,  the statement
\[
\SS_{i,p} \colon \left\{%
\begin{minipage}[c]{9.6cm}
\emph{the relation  $[a,b^w] = 1$ holds in $\tilde{G}$ for every couple $(a,b) \in \AA^2$\\ 
and every word $w = v_1 v_2 \cdots v_{i-1} u_i  w_{i+1}$ such that \\ 
$v_1 \in V_1$, \ldots, $v_{i-1} \in V_{i-1}$,  $w_{i+1} \in F_{i+1}$ and 
\\$u_i = t_{1,i}^{m_{1,i}}  \cdots t_{n_i,i}^{m_{n_i,i}}$ 
with 
$m_{1,i}^2 + \cdots  + m_{n_i,i}^2 \leq p$.
}
\end{minipage}
 \right \}
\]
By the definition 
\footnote{see equations \eqref{eq:Definition-W} and \eqref{eq:Definition-V-sub-i}}
of the set $V_i$,
statement $\SS_{i,p_0(i)}$ coincides with statement $\SS_{i+1}$,
a statement that holds by the outer induction.
This permits us to start the inner induction at $p = p_0(i)$.
Consider now a word
\[
u_i = t_{1,i}^{m_{1,i}}  \cdots t_{n_i,i}^{m_{n_i,i}} 
\text{ with } m_{1,i}^2 + \cdots + m_{n_i,i}^2 = p + 1 > p_0(i).
\]
Set  $x = \vartheta_i (\pi(u_i)) \in \Z^{n_i}$.
By Lemma \ref{lem:Geometric-lemma}, 
applied with
\[
\FF = \FF_i = \{ L_{j,i} \mid 1 \leq j \leq \ell_i \} \cup \{ -L_{j,i} \mid 1 \leq j \leq \ell_i \},
\]
and by the construction of $p_0(i)$, there exists an index $j$ and a sign $\varepsilon$
such that  $x + \varepsilon \cdot L_{j,i}$ is contained in the ball $\B_p \subset \Z^{n_i}$
with radius-squared $p$.
The finite set $L_{j,i}$ corresponds to $n_i$ relations in $\CC$ having the form
\begin{equation}
\label{eq:Relation-corresponding-L-sub-ji}
b = \prod\nolimits_{u \in F^0_i} \left(b^{\lambda(b;j,i,\hat{u})}\right)^u.
\end{equation}
In the following calculations, the couple  $(a,b) \in \AA$ is fixed.
Assume first that $\varepsilon$ equals $+ 1$.
Then the chain of relations
\[
[a,b^w] = 
\left[ a, \left(\prod\nolimits_{u \in F^o_i} b^{\lambda(b;j,i,\hat{u})\cdot u} \right)^w\right]
=
\prod\nolimits_{u \in F^o_i} \left[ a,  b^{\lambda(b;j,i,\hat{u})\cdot u \cdot w}\right]^{f(b,u)}
\] 
holds in the group $\tilde{G}$.
Here the second relation is a consequence of the commutator identity
\[
[a, bc] = a^{-1}c^{-1} b^{-1} \cdot abc = [a,b]^{c^a} \cdot [a,c], 
\]
and the conjugating factors $f(b,u)$ are certain elements in $\tilde{G}$ which need not concern us.
We next rewrite the conjugating words $u \cdot w$ \emph{modulo} the set of relators $\RR_Q$,
the defining relators of $Q$.
The following calculation holds in the free group $F$:
\[
u \cdot w = u \cdot v_1 v_2 \cdots v_{i-1} u_i w_{i+1} 
= 
v_1 v_2 \cdots v_{i-1} \cdot (u^{v_1 v_2 \cdots v_{i-1}} \cdot u_i )\cdot w_{i+1}. 
\]
The words $u$ and $u_i$ project onto elements $q$ and $q_i$ in $Q_i$ 
and this subgroup is central in $Q$ \emph{modulo} $Q_{i+1}$.
It follows that the word $u \cdot w$ is congruent, modulo the defining relators $\RR_Q$  of $Q$,
to a word of the form $v_1 v_2 \cdots v_{i-1} \cdot u'_i \cdot  w'_{i+1}$ 
with 
\[
u'_i = t_{1,i}^{m'_{1,i}}  \cdots t_{n_i,i}^{m'_{n_i,i}} \quad\text{and}\quad w'_{i+1} \in F_{i+1}.
\]
In addition,  
the lattice point $(m'_{1,i}, \ldots, m'_{n_i,i}) = (m_{1,i}, \ldots, m_{n_i,i}) + y$ with $y \in L_{j,i}$ 
lies in the ball $\B_p$.
The inductive hypothesis applies therefore to the word $v_1 v_2 \cdots v_{i-1} \cdot u'_i \cdot w'_{i+1}$.
It follows that each commutator
\[
\left[ a,  b^{\lambda(b;j,i,\hat{u})\cdot u \cdot w}\right]
\]
is trivial in $\tilde{G}$,
whence the commutator relation $[a,b^w] = 1$ holds in $\tilde{G}$.
\smallskip

Suppose now that $\varepsilon = - 1$.
Then the chain of relations
\[
[a,b^w] = 
\left[ \prod\nolimits_{u \in F^o_i} a^{\lambda(a;j,i,\hat{u})\cdot u}, b^w\right]
=
\prod\nolimits_{u \in F^o_i} \left[ a^{\lambda(a;j,i,\hat{u})\cdot u }, b^w \right]^{g(a,u)}.
\] 
holds in the group $\tilde{G}$.
Here the second relation is a consequence of the commutator identity
\[
[ab,c] = b^{-1}a^{-1} c^{-1} \cdot abc = [a,c]^b \cdot [b,c], 
\]
and the conjugating elements $g(a,u)$ are certain elements in $\tilde{G}$.
By rewriting the factors of the third term,  one arrives at the relation
\[
[a,b^w] = 
\prod\nolimits_{u \in F^o_i} 
\left[ a^{\lambda(a;j,i,\hat{u}) }, b^{w \cdot u^{-1}} \right]^{u\cdot g(a,u)}.
\]
As before, the aim is now to show that each factor on the right hand side is trivial in $\tilde{G}$.
To attain it one rewrites the conjugating words $w \cdot u^{-1}$ 
\emph{modulo} the set of relators $\RR_Q$ 
and uses then  the fact that the set $x - L_{j,i}$ 
lies inside the ball  $\B_p \subseteq \Z^{n_i}$ of radius $\sqrt{p}$.
Here are the details:
\[
w \cdot u^{-1} = 
v_1 v_2 \cdots v_{i-1} u_i w_{i+1} \cdot u^{-1}
= 
v_1 v_2 \cdots v_{i-1} \cdot  (u_i \cdot  w_{i+1} \cdot u^{-1}). 
\]
The factor in parentheses represents an element $q_i$ of $Q_i$;
this element can also be represented by a word of the form $u''_i \cdot  w'' _{i+1}$
with
\[
u''_i = t_{1,i}^{m''_{1,i}}  \cdots t_{n_i,i}^{m''_{n_i,i}} \quad\text{and}\quad w''_{i+1} \in F_{i+1}.
\]
Moreover,
$(m''_{1,i}, \ldots, m''_{n_i,i}) = (m_{1,i}, \ldots, m_{n_i,i}) - y$ for some $y \in L_{j,i}$.
It then follows as before, that $[a,b^w] = 1$ holds in $\tilde{G}$.
%
\subsubsection{Conclusion}
\label{sssec:Conclusion}
%
The previous calculation shows 
that implication $\SS_{i,p} \Rightarrow \SS_{i,p + 1}$  is valid for every integer $p \geq p_0(i)$.
Hence the interior induction allows one to deduce from statement $\SS_{i+1}$ the  statement
\[
\SS_{i,\infty} \colon \left\{%
\begin{minipage}[c]{9.6cm}
\emph{the relation  $[a,b^w] = 1$ holds in $\tilde{G}$ for every couple $(a,b) \in \AA^2$\\ 
and every  $w = v_1 v_2 \cdots v_{i-1} u_i  w_{i+1}$ with $v_1 \in V_1$, \ldots, \\ 
$v_{i-1} \in V_{i-1}$ , $u_i = t_{1,i}^{m_{1,i}}  \cdots t_{n_i,i}^{m_{n_i,i}}$ and  $w_{i+1} \in F_{i+1}$. 
}
\end{minipage}
 \right \}
\]
As the presentation of $\tilde{G}$ contains the relations $\RR_Q$ defining $Q$,
this statement $\SS_{i,\infty}$ is equivalent to $\SS_i$.
The exterior induction then allows us to conclude that statement $\SS_1$ is valid.
In view of the discussion following equation \eqref{eq:Definition-auxiliary-quotient},
statement $\SS_1$, finally, proves 
that the epimorphism $\kappa$,  figuring in Proposition \ref{prp:Kappa-is-iso},  is an isomorphism.
The proof of Proposition \ref{prp:Kappa-is-iso} 
and hence that of Theorem  \ref{thm:Construction-fp-group} is now complete.
%

%
%
\section{Examples and concluding remarks}
\label{sec:Examples-comments}
%
The purpose of this final section is twofold;
to illustrate the computation of $\Sigma^0$ by some select examples
and to discuss the necessity of the requirement of tameness.
%
\subsection{Examples of tame modules}
\label{ssec:Examples-tame-modules}
%
We describe some basic examples of tame modules over abelian groups and over nilpotent groups of class 2.
%
\subsubsection{$Q$ free abelian}
\label{sssec:Q-free-abelian}
Our first example goes back to the papers \cite{Bau73} and \cite{Rem73a} 
written by G. Baumslag and V. I. Remeslennikov, respectively.
Let $k \geq 1$ be an integer and $Q$ the free abelian group of rank $2k$
with basis $\BB = \{x_1, y_1, \ldots, x_k, y_k \}$. 
Define $A$ to be the cyclic $\Z{Q}$-module with defining (right) annihilator ideal
\begin{equation}
\label{eq:BR-ideal}
I = (1 + x_1-y_1) \cdot \Z{Q} + \cdots + (1 + x_k-y_k) \cdot \Z{Q}.
\end{equation}
Then $A$ is a tame $\Z{Q}$-module.

To see this,
choose $a = 1 + I \in A$ as generator of $A$ 
and consider a character $\chi \colon Q \to \R$ of $Q$.
Suppose first there exists an index $i \in \{1, \ldots, k\}$
so that $\chi$ assumes its minimum only once on the support $\{1, x_i, y_i \}$ of $1 + x_i - y_i$.
If the minimum occurs at 1 we rewrite the equation $a \cdot (1 + x_i - y_i) = 0$ in the form 
$a = a (-x_i + y_i)$  
and deduce from implication (iii) $\Rightarrow$ (i) in Proposition \ref{prp:Sigma0-criterion}
that $[\chi] \in \Sigma^0(Q;A)$ if $\chi(x_i) > 0$ and $\chi(y_i) > 0$.
If the minimum is taken on $x_i$ we use that $a$ is annihilated by 
$(1 + x_i - y_1) \cdot x_i^{-1} = 1 + x_i^{-1} - x_i^{-1}y_i$ 
and infer that $[\chi] \in \Sigma^0(Q;A)$ if $\chi(x_i^{-1}) > 0$ and $\chi(x_i^{-1}y_i) > 0$;
if it is taken on $y_i$ we argue similarly.
We conclude that $\chi$ can only lie outside of $\Sigma^0(Q;A)$ 
if it assumes its minimum at least twice on the support of each of the elements $1 + x_i - y_i$.
Two cases now arise:
if $\chi$ assumes its minimum only twice on the support of some element $1 + x_i - y_i$,
then $-\chi$ assumes its minimum only once on it, and so $[-\chi] \in \Sigma^0(Q;A)$;
if $\chi$ is constant on the support of each $1 + x_i - y_i$ 
then $\chi$ maps every generator of $Q$ to 0,
hence is the zero map and so does not represent a point of $S(Q)$.

The preceding calculations show 
that $\Sigma^0(Q;A) \cup -\Sigma^0(Q;A) = S(Q)$;
by Definition \ref{definition:Tame-module}
$A$ is thus tame with respect to the central series $\{1\} = Q_2 < Q_1 = Q$.
\begin{remarks}
\label{remarks:Verifying-tameness}
The above example brings to light two hallmarks of Theorem \ref{thm:Construction-fp-group}.

a) Typically, one will not be able to determine $\Sigma^0(Q;A)$ precisely.
Fortunately,  
Theorem \ref{thm:Construction-fp-group} does not presuppose the exact knowledge of $\Sigma^0(Q;A)$:
it suffices to find a lower bound $\Lambda \subseteq \Sigma^0(Q;A)$ with $\Lambda \cup -\Lambda = S(Q)$.

b) The module $A$ in the previous example is constructed by generators and defining annihilating elements.
By suitably choosing the annihilating elements one can arrange that the constructed module is tame;
it may, however, be far smaller then anticipated.
One way of avoiding a bad surprise is to select the annihilating element so 
that they can be interpreted as describing a localization of a known module.
This is the case with the annihilating ideal \eqref{eq:BR-ideal};
multiplication by each $1+x_i$ defines an injective endomorphism of the free cyclic module $\Z{P}$,
where $P$  is the free abelian group on $\{x_1, \ldots, x_k\}$.
The cyclic module  $A$ is therefore isomorphic to the localized polynomial ring 
\[
\Z[X_1, X_1^{-1}, (1 + X_1)^{-1}, \ldots, X_k, X_k^{-1}, (1 + X_k)^{-1}].
\]
This isomorphism shows, in particular,
that every wreath product $\Z \wr P = P \ltimes \Z{P}$ embeds into a finitely related metabelian group.
\end{remarks}
%
\subsubsection{$Q$ a Heisenberg group}
\label{sssec:Q-Heisenberg-group}
Concrete examples of torsion-free, nilpotent groups of class 2 are provided by Heisenberg groups.
Such a group is defined as follows:
\begin{definition}
\label{definition:Heisenberg-group}
A group $G$ is called a \emph{Heisenberg group of rank} $k$
if it admits a presentation with generators $\{x_1, y_1,  \cdots x_k, y_k, z\}$ 
and with defining relations
\begin{equation}
\begin{split} 
\label {eq:Heisenberg-group}
[x_i, x_j]  &=[y_i,y_j]=1 \text{ for all }1\le i < j\le k,\\
[x_i,y_j]  &=z^{\delta_{ij}}\text{ for all }1\le i, j\le k .
 \end{split}
 \end{equation}
\end{definition}
The element $z$ is of infinite order and generates the centre $Z$ of $Q$,
 and $G/Z$ is free abelian of rank $2k$.
 \smallskip
 
We next construct a tame module. 
Let $Q$ be a Heisenberg group of rank $k$
and define $A = \Z{Q}/I$ to be the cyclic $Q$-module with
\begin{equation}
\label{eq:Annihilating-ideal-Heisenberg}
I =  (1 + x_1-y_1) \cdot \Z{Q} + \cdots + (1 + x_k-y_k) \cdot \Z{Q} + (z-\ell) \cdot \Z{Q}
\end{equation}
where $\ell > 1$ is an integer. Set $a = 1 + I$ and $\AA = \{a\}$.
We assert that $A$ is a tame $\Z{Q}$-module 
with respect to the central series $\{1\} = Q_3 < Q_2 = Z< Q_1 = Q$.

To prove this claim we have to establish the inclusions
\begin{gather}
\Sigma^0(Q;A) \cup -\Sigma^0(Q;A) \supseteq S(Q,Z)  = S(Q),
\label{eq:Condition-1}\\
\Sigma^0(Z;\AA\cdot \Z{Z}) \cup \Sigma^0(Z;\AA \cdot \Z{Z}) \supseteq S(Z).
\label{eq:Condition-2}
\end{gather}
To verify the first inclusion, one can proceed as in the previous example;
to justify the second, 
one rewrites the relation $a \cdot (z-\ell) = 0$ in the form $a \cdot (1-\ell\cdot z^{-1}) = 0$
and infers then from Proposition \ref{prp:Sigma0-criterion} 
that the character $\chi \colon Z = \gp(z) \to \R$ with $\chi(z^{-1}) = 1$ 
represents a point of $\Sigma^0(Z;\AA \cdot \Z{Z})$.

\begin{remark}
\label{remark:Structure.tame-module}
The relations imposed on the module $A$ in the preceding example can be interpreted as arising from a sequence of ascending HNN-extensions (cf.\;\cite[Sect.\;3]{GrSt11}). 
One can see in this way 
that the semi-direct product $Q \ltimes A$ contains the wreath product $\Z \wr \gp(\{x_1, \ldots, x_k\})$ .
In Theorem 1 of \cite{RoSt82} a different route is chosen: 
there $k = 1$, 
and the element $y_1 + x_1 -x_1^2$ is chosen in place of the element $1 + x_1 - y_1$.
One then sees by calculation 
that $A$ is isomorphic to $(\Z[1/\ell] P )^2$
as module over the subgroup $P = \gp(z, y_1)$ of $Q$.
\end{remark}

%
\subsection{On the necessity of the requirement of tameness}
\label{ssec:Necessity-tameness-module}
%
In Section \ref{sec:Construction-groups},
a geometric method is used to establish the finite presentability of $G = Q \ltimes A$. 
This method goes back to the paper \cite{BiSt80} by R. Bieri and the author.
In the cited paper, the following result is proved:
\begin{thm}[\protect{\cite[Thm.\;A]{BiSt80}}]
\label{thm:BiSt80}
Let $G$ be a finitely generated group containing an abelian normal subgroup $A$ 
with abelian factor group $Q = G/A$. View $A$ as a right $\Z{Q}$-module via conjugation.
Then $G$ is finitely presentable if, and only if, 
\begin{equation}
\label{eq:Tameness-abelian}
\Sigma^0(Q;A) \cup -\Sigma^0(Q;A) = S(Q).
\end{equation}
\end{thm}
The requirement \eqref{eq:Tameness-abelian} is nothing but the condition 
that the module $A$ be \emph{tame} (with respect to the central series $\{1\} = Q_2 < Q_1 = Q$).
The theorem reveals that this condition is not merely sufficient, but also necessary.
In addition, the group $G$ need not be the semi-direct product $Q \ltimes A$.
One can modify the proof of Theorem \eqref{thm:Construction-fp-group}
so that non-split extensions are also covered (see \cite[§4 and §5]{Str81b}),
but the question to what extent the condition of being tame is necessary seems to be of greater significance
then this generalization.
Prior to discussing it,
I explain how the necessity of $A$ being tame is established in the case of metabelian groups.

One starts out with a finitely presented group $G$ and a non-zero character $\chi \colon G \to \R$.
One then expresses $N = \ker \chi$ as a free product with amalgamation, say $S_1 \star_{S_0} S_2$
(see \cite[Sects.\;4.2--4.4]{BiSt80}).
Such a product typically contains non-abelian free subgroups.
If $G$ does not contain  a non-abelian free subgroup, for instance because it is soluble,
the free product with amalgamation must be degenerate. 
By analyzing this degeneration, one finds that the module $A$ has to be tame over the group $Q$
(see \cite[Sects.\;4.5--4.7]{BiSt80}).
%
\subsubsection{The invariant $\Sigma$}
\label{sssec:Invariant-Sigma}
%
The sketched argument can be generalized.
The representation of $\ker \chi$ as a free product with amalgamation 
is available for an arbitrary finitely presented group.
In the next step one needs an invariant 
that is capable of recording that the representation is degenerate 
if the group $G$ does not contain a non-abelian free subgroup.
Such an invariant has been introduced and studied in \cite{BNS}.
Its definition is this:
\begin{equation}
\label{eq:Definition-Sigma-BNS}
\Sigma(G) = \Sigma_{G'}(G) =  \{ [\chi] \in S(G) \mid G' \text{ fg over a fg submonoid of } G_\chi \}.
\end{equation}
Here $G$ is a finitely generated group and its derived group  $G'$ is viewed as a $G$-group via conjugation;
the definitions of the sphere $S(G)$ and of the submonoid $G_\chi$ are as in section 
\ref{sssec:Sigma0-definition}. Finally,  ``fg'' is short for ``finitely generated''.

With the help of the invariant $\Sigma$ the impact of finite presentability can then be expressed as follows:
\begin{thm}[\protect{\cite[Thm.\;C]{BNS}}]
\label{thm:Necessary-condition-BNS}
If $G$ is a finitely presented group which contains no non-abelian free subgroup
then
\begin{equation}
\label{eq:Necessary-condition-BNS}
\Sigma(G) \cup -\Sigma(G) = S(G).
\end{equation}
\end{thm}
\subsubsection{A necessary condition for the finite presentation}
\label{sssec:Necessary-condition-for-AP-groups}
I  begin with a result 
that compares the invariants $\Sigma(G)$ and $\Sigma^0(Q;A)$ in a special case:
\begin{prp}
\label{prp:Relating-invariants}
Let $G$ be an extension of an abelian normal subgroup $A$ by a polycyclic group $Q$;
let $\pi \colon G \epi Q$ denote the associated projection.
Then  the biimplication
\begin{equation}
[\chi] \in \Sigma^0(Q;A) \Longleftrightarrow [\chi \circ \pi] \in \Sigma(G)
\end{equation}
holds for every non-zero character  $\chi \colon Q \to \R$-
\end{prp}

\begin{proof}
Note first that $A$ is a finitely generated $\Z{Q}$-module,
for $G$ is finitely generated and the polycyclic group $Q$  is finitely presentable.
Choose a finitely generated subgroup $H$ of $G'$ that projects onto $Q'$ 
and let $\HH \subset  H$ be a finite set of generators.

Assume first that $A$ is contained in $G'$.
If $[\chi] \in \Sigma^0(Q;A)$ 
then $A$ is finitely generated over $Q_\chi$, say by $\AA$;
Proposition \ref{prp:Existence-finitely-generated-subring} then shows
that $\AA$ generates $A$ over a finitely generated submonoid $M \subset Q_\chi$.
Pick a a finitely generated submonoid $\tilde{M} \subset G_{[\chi \circ \pi]}$ that projects onto $M$.
Then $\AA \cup \HH$ generates $G'$ over $\tilde{M}$ and so $[\chi \circ \pi] \in \Sigma(G)$.
Conversely, if $[\chi \circ \pi] \in \Sigma(G)$,
then $G'$ is finitely generated over a finitely generated submonoid $\md(\YY)$ of  $G_{[\chi \circ \pi]}$.
By Lemma \ref{lem:BNS} below,
the module $A$ will then be finitely generated over a submonoid of $Q_\chi$ 
and so $[\chi] \in \Sigma^0(Q;A)$.

Suppose now $A$ is not contained in $G'$ and set $A_1 = A \cap G'$.
The factor group $A/A_1 = A /(A \cap G')$ is isomorphic to $(A \cdot G')/G'$ 
and so a finitely generated abelian group.
It follows that $G/A_1$ is polycyclic and so the previous reasoning applies to $A_1$.
In addition, Corollary \ref{crl:Invariants-ses-over-polycyclic-group} shows that $\Sigma^0(Q;A) = \Sigma^0(Q;A_1)$.
\end{proof}

It remains to prove a technical result (cf.\;\cite[Lemma 3.5]{BNS}):
\begin{lem}
\label{lem:BNS}
Let $G$ be an extension of an abelian normal subgroup $A \subset G'$ by a polycyclic group $Q$.
If $G'$ is finitely generated over a finitely generated submonoid $\md(\YY)$ of $G$ 
then $A$ is finitely generated over the monoid $\md(Q', \pi(\YY))$.
\end{lem}
\newpage

\begin{proof}
Let $\HH \subset G'$ be a finite set which generates $G'$ over the monoid $\md(\YY) \subset G$.
As $Q'$ is finitely generated, 
we can assume that $\pi(\HH)$ generates $Q'$. 
Set  $H  = \gp(\HH)$.

Since $H$ is finitely generated and  $Q'$ is finitely presentable, 
the kernel 
\[
B = \ker (\pi_{|H}) = A \cap H
\] 
is the normal closure $\gp_H(\BB)$ of a finite subset $\BB \subset H$.
In addition. there exists for every $h \in H$ and every $y \in \YY$ 
an element $a(h,y) \in A$ and an element $u(h,y) \in H$ satisfying the equation
\begin{equation}
\label{eq:Definition-a(h,y)}
h^y = a(h, y) \cdot u(h,y).
\end{equation}
We claim that $\AA = \BB \cup \{a(h,y) \mid h \in \HH \cup \HH^{-1} \text{ and } y \in \YY \}$
generates $A$ over the monoid $\md(\HH \cup \HH^{-1} \cup \YY)$.

To prove this claim we set  $A_1 = \gp(\AA^{\,\md(\HH \cup \HH^{-1} \cup \YY)})$ 
and show first that the product $A_1 \cdot H$ contains $G'$. 
Since $A_1$ is normalized by each monoid generator of $H$ this product is a subgroup of $A$; 
so it suffices to verify that each  generator $h^w$ of $G' = \gp(\HH^{\md(\YY)})$ can be written 
as a product $a_1 \cdot h_1$ with $a_1 \in A_1$ and $h_1 \in H$. 
This is clearly possible if $w = 1$ and holds for $w \in \YY$ by equation \eqref{eq:Definition-a(h,y)}.
If $w = w' y$ 
we may assume inductively that $h^{w'} = a_1\cdot h_1$, 
whence  $h^w= a_1^y \cdot h_1^y$. 
Now $a_1^y$ is in $A_1$ by the very definition of $A_1$,
while $h_1$ is a product with factors in $\HH \cup \HH^{-1}$ 
and so $h_1^y$ is in $A_1 \cdot H$ by equation \eqref{eq:Definition-a(h,y)}.

The claim that $A_1 = A$ now follows like this: 
As $A$ is contained in $G'$ by hypothesis, 
every element $a \in A$ can be written as a product $a_1 \cdot h_1$ 
with $a_1 \in A_1$ and $b_1 \in H$. 
Then $h_1$  is in $A \cap H = B$, 
so lies in $A_1$ by the definitions of $\AA$, of $B = gp_H(\BB)$ and of $A_1$,
whence $a = a_1 \cdot  h_1 \in A_1$.
\end{proof}

Proposition \ref{prp:Relating-invariants} leads immediately to a necessary condition 
for the finite presentability of an abelian-by-polycyclic group:
\begin{crl}
\label{crl:Necessary-condition-AP-group}
Let $G$ an extension of an abelian normal subgroup $A$ by a polycyclic group $Q$.
If $G$ has a finite presentation then
\begin{equation}
\label{eq:Necessary-condition-AP-group}
\Sigma(Q;A) \cup -\Sigma(Q;A) = S(Q).
\end{equation}
\end{crl}
\begin{proof}
Let $\pi \colon G \epi Q$ be the projection associated to the extension $A \triangleleft  G \epi Q$
and let $\chi \colon Q \to \R$ be a non-trivial character.
By Theorem \ref{thm:Necessary-condition-BNS} at least one of the antipodal characters $\pm(\chi \circ \pi)$ represents then a point in $\Sigma(G)$
whence Proposition  \ref{prp:Relating-invariants} allows one to infer 
that $[\chi] \in \Sigma^0(Q;A) \cup -\Sigma^0(Q;A)$.
\end{proof}

\begin{remarks}
\label{remarks:Necessary-condition-fp-AP-groups}
a) Corollary \ref{crl:Necessary-condition-AP-group} is an unpublished result of R. Bieri and the author.

b) Let $Q$ be a nilpotent group of class 2 
and let $Q_3 < Q_2< Q_1 = Q$ be the lower central series of $Q$.
Let $A$ be a finitely generated $\Z{Q}$-module, generated by the finite set $\AA$.
If $G = Q \ltimes A$ is finitely presentable then
$\Sigma^0(Q; A) \cup - \Sigma^0(Q;A) = S(Q)$ by Corollary \ref{crl:Necessary-condition-AP-group}.
The module $A$, however, is only tame with respect to the lower central series
if, in addition,  
\[
\Sigma^0(Q'; \AA \cdot \Z{Q'}) \cup -  \Sigma^0(Q'; \AA \cdot \Z{Q'}) = S(Q)
\]
or, equivalently, if the metabelian group $Q' \ltimes (\AA \cdot \Z{Q'})$ admits a finite presentation.
As of writing,
it is unknown to what extent this additional condition is implied by the finite presentability of $G$.
\end{remarks}

\def\cprime{$'$}

\medskip

Ralph Strebel, Department of Mathematics, Chemin du Mus{\'e}e 23, \\University of Fribourg,
CH-1700 Fribourg, Switzerland

E-mail: ralph.strebel@unifr.ch

%
\end{document}